\documentclass[11pt,reqno]{amsart}
\usepackage{amsmath}
\usepackage{cases}
\usepackage{amssymb, latexsym, amsmath}
\usepackage[colorlinks,
linkcolor=red,
anchorcolor=magenta,
citecolor=blue]{hyperref}
\usepackage{hyperref}
\usepackage{mathrsfs}
\usepackage{amscd}
\usepackage{mathtools}
\usepackage{amsbsy}
\usepackage{graphicx}
\usepackage{color}
\usepackage[all]{xy}
\usepackage{bigstrut}
\usepackage{geometry}
\begin{document}
\newtheorem{theoreme}{Theorem}
\newtheorem{lemma}{Lemma}[section]
\newtheorem{proposition}[lemma]{Proposition}
\newtheorem{corollary}[lemma]{Corollary}
\newtheorem{definition}[lemma]{Definition}
\newtheorem{conjecture}[lemma]{Conjecture}
\newtheorem{remark}[lemma]{Remark}
\newtheorem{exe}{Exercise}
\newtheorem{theorem}[lemma]{Theorem}
\theoremstyle{definition}
\numberwithin{equation}{section}
\newcommand{\R}{\mathbb R}
\newcommand{\TT}{\mathbb T}
\newcommand{\Z}{\mathbb Z}
\newcommand{\N}{\mathbb N}
\newcommand{\Q}{\mathbb Q}
\newcommand{\Var}{\operatorname{Var}}
\newcommand{\tr}{\operatorname{tr}}
\newcommand{\supp}{\operatorname{Supp}}
\newcommand{\intinf}{\int_{-\infty}^\infty}
\newcommand{\me}{\mathrm{e}}
\newcommand{\mi}{\mathrm{i}}
\newcommand{\dif}{\mathrm{d}}
\newcommand{\beq}{\begin{equation}}
\newcommand{\eeq}{\end{equation}}
\newcommand{\ben}{\begin{eqnarray}}
\newcommand{\een}{\end{eqnarray}}
\newcommand{\beno}{\begin{eqnarray*}}
\newcommand{\eeno}{\end{eqnarray*}}

\def\d{\delta}
\def\a{\alpha}
\def\e{\varepsilon}
\def\ld{\lambda}
\def\p{\partial}
\def\v{\varphi}
\newcommand{\D}{\Delta}
\newcommand{\Ld}{\Lambda}
\newcommand{\n}{\nabla}
\newcommand{\GG}{\text{g}}

\title[Local smoothing of FIO]
{Improved variable coefficient
	square functions and
	local smoothing of Fourier integral operators }

\author[Gao]{Chuanwei Gao}
\address{\hskip-1.15em Chuanwei Gao:
	\hfill\newline The Graduate School of China Academy of Engineering Physics,
	\hfill\newline P. O. Box 8009,\ Beijing,\ China,\ 100088,}
\email{canvee@163.com}

\author[Miao]{Changxing Miao}
\address{\hskip-1.15em Changxing Miao:
	\hfill\newline Institute of Applied Physics and Computational
	Mathematics,
	\hfill\newline P. O. Box 8009,\ Beijing,\ China,\ 100088,}
\email{miao\_changxing@iapcm.ac.cn}

\author[Yang]{Jianwei-Urbain Yang}
\address{\hskip-1.15em Jianwei-Urbain Yang£º
	\hfill\newline Department of Mathematics,
	\hfill\newline Beijing Institute of Technology,
	\hfill\newline Beijing 100081,\ P. R.  China}
\email{jw-urbain.yang@bit.edu.cn}


\begin{abstract}
	We establish certain square function
	estimates for a class of
	oscillatory integral operators with homogeneous phase functions.
	These results are employed to deduce
	a refinement of
	a previous result of Mockenhaupt Seeger and Sogge \cite{MSS-jams}
	on the local smoothing property for Fourier integral operators,
	which arise naturally in the study
	 of wave equations on
	compact Riemannian manifolds.
	The proof is an adaptation of the bilinear approach of Tao and Vargas \cite{Tao-Vargas-II},
and  based on bilinear oscillatory integral estimates of Lee \cite{Lee-JFA}.
\end{abstract}
\maketitle

\section{Introduction}
\label{sect:introd}
\subsection{Motivation and background}
The purpose of this paper is to study the local smoothing property for
a certain class of Fourier integral operators
acting on locally $L^p-$integrable functions defined on paracompact manifolds, which in the terminology of \cite{MSS-jams} satisfy the \emph{cinematic curvature} conditions.
We obtain improvements upon the known
$L^p\to L^p$ regularity results  for these operators for relatively small $p$, which were established previously in \cite{MSS-jams}.

Let $Z$ and $Y$ be smooth  paracompact manifolds with ${\rm dim} \;Z=n+1$
and ${\rm dim}\; Y=n\geqslant 2$, respectively. A Fourier integral operators  $\mathscr{F}\in I^{\sigma-\frac{1}{4}}(Z,Y;\mathscr{C})$ is said to
satisfy the cinematic curvature
condition  in the terminology of \cite{MSS-jams} as follows.
First of all, a Fourier integral operator $\mathscr F$ is determined globally by the canonical relation  $\mathscr{C}$, which  is a $2n+1$ dimensional closed homogeneous, conic
Lagrangian submanifold of $T^*Z\setminus 0\times T^*Y\setminus 0$
with respect to the symplectic form
$d\zeta\wedge dz-d\eta\wedge dy$.
Next, for a given $z_0\in Z$,
we consider the following diagram,
\begin{equation}
\label{jsdkmkcsd}
\xymatrix{
	&\mathscr{C}\ar[d]_{\varPi_{Z}} \ar[dl]_{\varPi_{T^{*}Y}}\ar[dr]^{\varPi_{T^{*}_{z_0}Z}}	
	&\\ T^{*}Y\setminus 0
	&Z
	& T^{*}_{z_0}Z\setminus0}
\end{equation}
where $\varPi_{T^{*}Y}$, $\varPi_{T^{*}_{z_0}Z}$
and $\varPi_{Z}$ are projections from $\mathscr{C}$ to $T^{*}Y\setminus 0$, $T^*_{z_0}Z\setminus 0$ and $Z$, respectively.
The first part of the cinematic curvature condition is an assumption on the nondegeneracy of
the first two projections in  \eqref{jsdkmkcsd},  requiring that
both of them are submersions,
\begin{equation}
\label{dkmkldsc}
{\rm rank}\; d\varPi_{T^{*}Y}\equiv 2n,
\end{equation}
\begin{equation}
\label{vmklsdmks}
{\rm rank}\; d \varPi_{Z}\equiv n+1.
\end{equation}
The second part of this condition concerns the curvature properties of
the image of $\varPi_{T^{*}_{z_0}Z}$ from $\mathscr C$, denoted by $\Gamma_{z_0}=\varPi_{T^{*}_{z_0}Z}(\mathscr C)$, as
an immersed hypersurfaces in the cotangent space $T^*_{z_0}Z\setminus 0$.
Then,
as a consequence of \eqref{dkmkldsc}\eqref{vmklsdmks} and the  homogeneity of $\mathscr C$,
$\Gamma_{z_0}$ is a smooth conic $n-$dimensional hypersurface in $T^*_{z_0}Z\setminus0$.
We shall impose the \emph{cone condition} on $\mathscr C$ by requiring that for every $\zeta\in \Gamma_{z_0}$,
there are $n-1$ principal curvatures which do not vanish.
We say thus  $\mathscr C$
satisfies the cinematic curvature condition,
if it satisfies \eqref{dkmkldsc}\eqref{vmklsdmks} and the cone condition.

One of the motivation on
the study of this kind of operators
may date back to the work of Stein \cite{Stein76}, investigating  the $L^p\to L^p$ boundness of maximal average
operators on the sphere, which is in  close connection to the pointwise convergence properties of wave equations and some of other topics in harmonic analysis on symmetric spaces.
When $n\geqslant 3$, Stein proved the optimal $L^p\to L^p$ estimate in \cite{Stein76} for $\frac{n}{n-1}<p\leqslant\infty$, while the two dimensional case was left open due to  impossibility of using the $L^2-$theory of Fourier transform. This problem in the two dimensional case was settled a couple of years later by Bourgain
\cite{Bo86} and generalized  to the variable coefficient setting by Sogge \cite{Sogge91}, where a primary version of local smoothing property for wave equations was employed in $(1+2)$ dimension to prove the $L^p\to L^p$ boundness of circular maximal averaging operators. A much more profound connection
between the circular maximal functions and the local smoothing properties of wave equations was found and clarified
later in \cite{MSS-Ann},
where the authors provided a much simplified proof of
Bourgain's circular maximal theorem based on a more sophisticated  local smoothing property of 2D wave equations in Euclidean space, a greatly improved result in the constant variable setting compared to \cite{Sogge91}.
These results were  generalized to an abstract theory concerning Fourier integral operators fulfilling the cinematic curvature condition later by Mockenhaupt, Seeger and Sogge \cite{MSS-jams} for all dimensions $n\geqslant 2$.

It is conjectured in \cite{Sogge91} for the wave equation on Euclidean space, and in \cite{MSS-jams} for Fourier integral operators satisfying the cinematic curvature condition that
for $p\geqslant 2n/(n-1)$, there should
always be an order
$1/p$ local smoothing property for these operators, \emph{i.e.},  whenever $\sigma<-\frac{n-1}{2}+\frac{n}{p}$, the operators $\mathscr{F}\in I^{\sigma-\frac{1}{4}}(Z,Y;\mathscr{C})$ are bounded from $L^p_{\rm comp}(Y)$ to $L^p_{{\rm loc}}(Z)$.

This is referred  as the \emph{local smoothing conjecture} in the literature
of modern Fourier analysis and has attracted extensive works of study. If this conjecture would have been proved,
then it would imply positive answers to a number of open conjectures concerning fundamental
problems of harmonic analysis and geometric measure theory, including
(maximal) Bochner-Riesz means, Fourier restriction theorem, Kakeya and Nikodym maximal function estimates as well as Hausdorff dimensions of Besicovitch sets in all dimensions. See the works \cite{Bo91gafa,Bo91St,Tao99Duke,Wolff99,KaTa02}
and references therein.

Proving the sharp $L^p-L^p$ local smoothing estimates appears to be very difficult, even for the wave equation posed on Euclidean spaces. Using microlocal analysis and
$L^2\to L^p$ local smoothing, or rather Strichartz's estimates in modern terminology, which turns out to be much easier to prove,  Mockenhaupt, Seeger and Sogge \cite{MSS-jams} demonstrated certain $L^{q_n}-$square function inequality
with $q_n:=2\frac{n+1}{n-1}$
for $n\geqslant 3$ (which coincides with the exponent of symmetric Strichartz's estimate for wave equations) and  established an $L^4-$square function inequality in dimension two by
exploring orthogonality in circular directions via  bilinear $L^2$ geometric approach. Combined with a variable coefficient version
of C\'ordoba's Kakeya maximal operator
along the direction of light rays, these square function estimates yielded
certain \emph{non-sharp}, $L^p\to L^p$ local smoothing estimates for  Fourier integral operators of this kind.

Concerning the question of the constant coefficient  wave equations, these local smoothing properties are  known to be deduced from the cone multiplier estimates and there are  subsequent improved results  by many authors mainly on the $(1+2)-$dimensions: see Bourgain \cite{Bourgain95}, the first improvement for the cone multiplier; Tao and Vargas \cite{Tao-Vargas-II}, Garrig\'os and Seeger \cite{GaSe09} using the bilinear method based on bilinear restriction estimates \cite{Wolff};  More recently, Lee  \cite{Lee18P} further improved the $L^4$ local smoothing estimate using  $\ell^2-$decoupling inequality of Bourgain and Demeter \cite{BoDe2015}.
All of these works are away from optimal with respect to the regularity and concern $L^p\to L^p$ local smoothing estimates for $p<q_n$.

The first sharp $L^p\to L^p$-local smoothing estimate was obtained by Wolff \cite{Wolff00} in $(1+2)-$dimensions for  $p>74$ and was extended to the higher dimensional cases by \L aba and Wolff \cite{LaWo02}. The borderline for this range of $p$  was refreshed later by Garrig\'os, Schlag and Seeger \cite{GaSchSe} and ultimately improved down to the Strichartz exponent $p\geqslant q_n$  by Bourgain and Demeter \cite{BoDe2015} via their celebrated $\ell^2-$decoupling theorem. Lee and Vargas \cite{LeVa12}, using the Bourgain and Guth multilinear approach in \cite{BoGu11} on the basis of the multilinear restriction theorem of Bennett, Carbery and Tao \cite{BCT06} obtained the sharp  local smoothing estimates for $p=3$.

Notice that none of the above works deals with possible improvements for
abstract theory of \cite{MSS-jams} at the generality of Fourier integral operators within the framework
set up at the very beginning of this section.
For $p\geqslant q_n$, Beltran, Hickman and Sogge \cite{BelHicSog18P} established the sharp $L^p\to L^p$ local smoothing estimates for Fourier integral operators of this kind by extending
Bourgain-Demeter's decompling inequality
to the variable coefficient setting.
In addition, they also construct
examples to show the optimality of their results in \emph{odd} dimensions at the level of such generality.
In particular, one can not expect an order $1/p$ local smoothing estimates for all $p$
between $\frac{2n}{n-1}$ and $q_n$.
It is conjectured in \cite{BelHicSog18P,BelHicSog18P-surv} that if $n\geqslant 2$ is even, there should be optimal local smoothing estimates for Fourier integral operators satisfying the cinematic curvature conditions whenever $p\geqslant \frac{2(n+2)}{n}$.
Furthermore, it is also conjectured in \cite{BelHicSog18P} that
if one imposes a {convexity} condition on the cone $\Gamma_{z_0}$ with a requirement that $\Gamma_{z_0}$ always has $n-1$ \emph{positive} principal curvatures,
then the optimal local smoothing property would be able to hold for these operators whenever $p\geqslant p_{n,+}$, with
\[
p_{n,+}:=
\begin{cases}
\frac{2(3n+1)}{3n-3},\quad \text{if}\; n\; \text{is odd},\\
\frac{2(3n+2)}{3n-2}, \quad \text{if}\; n\;\text{is even}.
\end{cases}
\]
When dimension $n=2$, by interpolation with the trivial $L^2-$endpoint inequality, Beltran-Hickman-Sogge's sharp  $L^6-$estimate in  \cite{BelHicSog18P}
only recovers the $1/8-$result for the $L^4\to L^4$ local smoothing estimate of Mockenhaupt, Seeger and Sogge \cite{MSS-jams}.

To our knowledge,  this is the best results so far at such a level of generality for the abstract theory of Fourier integral operators. In particular, when $n\geqslant 2$ and $p<q_n$,
the above two conjectures formulated in \cite{BelHicSog18P,BelHicSog18P-surv} are completely open and this paper is intended to provide some partial answers to them.
Now, we state
our main result.
\begin{theoreme}
	\label{main}
	Let $Z$ and $Y$ be  smooth paracompact manifolds of dimension
	$3$ and $2$ respectively.
	Suppose that $\mathscr F\in I^{\sigma-\frac{1}{4}}(Z,Y;\mathscr C)$,
	is a Fourier integral operator of order $\sigma$, whose canonical relation $\mathscr C$ satisfies the cinematic curvature condition.
    Then the following estimate holds
    	\begin{equation}
    \label{eq:main}
    \bigl\|\mathscr{ F }f\bigr\|_{L^p_{\rm loc}(Z)}\leqslant C\|f\|_{L^p_{\rm comp}(Y)}, \quad
    \text{for all}\quad\sigma<-\bar{s}_p+\varepsilon(p),
    \end{equation}
    where $2\leqslant p\leqslant 6$, $\bar{s}_p=\frac{1}{2}-\frac{1}{p}$
    and
    \begin{equation}
    \label{eq:sigma-n}
        \varepsilon(p)=
    \begin{cases}
    \frac{3}{16}-\frac{1}{8p},\;&\text{if}\quad \frac{10}{3}\leqslant p\leqslant 6,\\
    \frac{3}{8}-\frac{3}{4p},\;& \text{if}\quad
    2\leqslant p\leqslant \frac{10}{3}
    \end{cases}
    \end{equation}
\end{theoreme}
It is natural to compare our result with the previous one in \cite{MSS-jams}.
From \eqref{eq:main}, we obtain the
$L^4\to L^4$ estimate for $\mathscr F$
whenever $\sigma<-\frac{3}{32}$, which brings in an improvement with respect to the regularity level of order $\frac{1}{32}$.

Now we turns to a couple of corollaries of Theorem \ref{main}. As in \cite{BelHicSog18P-surv}, the local smoothing  theory of Fourier integral operators has many applications to various of interesting problems in harmonic analysis. In this paper, we include three examples of them below which follows  immediately from Theorem \ref{main} and we omit the proofs. One may consult the beautifully written self-contained article \cite{BelHicSog18P-surv} for more details.

The first application concerns the maximal function estimate of Fourier integral operators.
If we write $z=(x,t)$ and let $\mathscr{ F}_tf(x)=\mathscr Ff(x,t)$, then we have the following maximal theorem under the above assumptions in Theorem \ref{main}.
\begin{corollary}
	Assume that $\mathscr F\in I^{\sigma-1/4}(Z,Y;\mathscr C)$ is a Fourier integral operator satisfying all the same conditions in Theorem \ref{main}.
	Then, if $I\subset\R$ is a compact interval and $Z=X\times I$ such that $X$ and $Y$ are assumed to be compact, then we have
	\begin{equation}
		\label{eq:maximal}
		\|\sup_{t\in I}|\mathscr{F}_{t}f(x)|\|_{L^p(X)}\leqslant C\|f\|_{L^p(Y)},
	\end{equation}
	whenever $\sigma<-\bar{s}_{p}-(1/p-\varepsilon(p))$ for all $2\leqslant p\leqslant 6$.
\end{corollary}
 Notice that it is conjectured in \cite{MSS-jams} that
the maximal estimate \eqref{eq:maximal}
should hold as long as $\sigma<-\bar{s}_{p}$ for all $p\geqslant \frac{2n}{n-1}$. The sharp result for $p\geqslant 6$ is obtained in \cite{BelHicSog18P} and
our result provides certain improvement
of  Corollary 6.3 in \cite{MSS-jams}
for $p\leqslant 6$.

The second applications of Theorem \ref{main} is related to the regularity properties of wave equations on smooth compact Riemmanian manifold.
Let $M$ be a smooth compact manifold without boundary of dimension $n$, equipped with a Riemmanian metric $\GG$ and consider
the Cauchy problem
\begin{equation}
\label{eq:0}
\left\{ \begin{aligned}
(\partial/\partial t)^2-\Delta_\GG ) &u(t,x)= 0,\,(t,x)\in \R\times M,\\
u(0,x)=f(x),&\;
\partial_tu(0,x)=h(x),\\
\end{aligned} \right.
\end{equation}
where $\Delta_\GG$ is the
Beltrami-Laplacian associated to a metric $\GG$. It is a well-known fact that the solution $u$ to this Cauchy problem can be written as
\begin{equation}
\label{eq:wave-solu}
u(x,t)=\mathscr{ F}_0f(x,t)+\mathscr{F}_1h(x,t),
\end{equation}
where $\mathscr{F}_j\in I^{j-1/4}(M\times \R,M;\mathscr C)$ with
\[	\mathscr{C}=\left\{(x,t,\xi,\tau,y,\eta):(x,\xi)=\chi_t(y,\eta),
	\tau=\pm\sqrt{\sum\GG^{jk}\xi_j\xi_k}\right\},
\]
where $\chi_t:T^*M\setminus0\times T^*M\setminus 0$ is given by flowing for time $t$ along the Hamilton vector field $H$ associated to $\sqrt{\sum\GG^{jk}\xi_j\xi_k}$.
 As a consequence, the convexity condition is automatically verified by $\mathscr C$. The following result which improved  Corollary 6.4 of \cite{MSS-jams} is readily deduced from Theorem \ref{main}.
\begin{corollary}
	Let $u$ be the solution to the Cauchy problem
	\eqref{eq:0}.
	If $I\subset\R$ is a compact interval and $0<\delta<\varepsilon(p)$ with $\varepsilon(p)$ is given by \eqref{eq:sigma-n}, then we have
	\begin{equation}
	\label{eq:wave}
	\|u\|_{L^p_{\alpha-\bar{s}_{p}+\delta}(M\times I)}
	\leqslant C
	\Bigl(
	\|f\|_{L^p_\alpha(M)}+\|h\|_{L^p_{\alpha-1}(M)}
	\Bigr), \;2\leqslant p\leqslant 6.
	\end{equation}
\end{corollary}

The third application of our
main theorem is about
an averaging operator over smooth curves, a question raised by Sogge \cite{Sogge91} and later extended to the higher dimensions  over smooth hypersurface in $\R^n$ by Schlag and Sogge \cite{SchSog97}. Let
$\varSigma_{x,t}\subset\R^2$
be a smooth curve depending smoothly
on the parameters
$(x,t)\in\R^2\times[1,2]$ and
$\dif \sigma_{x,t}$ denotes the
normalized Lebesgue
measure on $\varSigma_{x,t}$.
Following the notations in \cite{Sogge91,SchSog97}, we may assume
$\varSigma_{x,t}=\{y:\Phi(x,y)=t\}$
where $\Phi(x,y)\in C^\infty(\R^2\times\R^2)$ such that its Monge-Ampere determinant is non-singular
\begin{equation}
\label{eq:MA}
{\rm det}
\begin{pmatrix}
0 & \partial\Phi/\partial x  \\
\partial\Phi/\partial y & \partial^2\Phi/\partial x\partial y \\
\end{pmatrix}
\neq 0,\quad \text{when}\;
\Phi(x,y)=t.
\end{equation}
This  is referred to
as Stein-Phong's rotational
curvature condition.

Define the averaging operator by
\begin{equation}
\label{eq:A}
A f(x,t)
=\int_{\varSigma_{x,t}}
f(y)a(x,y)\dif \sigma_{x,t}(y),
\end{equation}
where $a(x,y)$ is a smooth function with compact support in $\R^2\times\R^2$.

If we let
$\mathscr{F}:f(x)\mapsto
Af(x,t)$, then $\mathscr F$
 is a Fourier integral operator of order $-1/2$ with canonical relation
 given by
 \begin{equation}
 \label{eq:cano}
 \mathscr{C}=\{(x,t,\xi,\tau,y,\eta):(x,\xi)=\chi_t(y,\eta),\tau=q(x,t,\xi)\},
 \end{equation}
where $\chi_t$ is a local symplectomorphism, and
the function
$q$ is homogeneous of degree one in $\xi$ and smooth away from $\xi=0$.
Moreover, the rotational curvature condition holds if and only if
\begin{equation}
\label{eq:qq}
\left\{
\begin{aligned}
q(x,\Phi(x,y),\Phi'_x(x,y))&\equiv 1\\
\text{corank}\;\; q''_{\xi\xi}&\equiv 1.\\
\end{aligned} \right.
\end{equation}
In particular, the cinematic curvature condition is fulfilled and one has the following result by Theorem \ref{main}.
\begin{corollary}
	Let $A_t$ be an averaging operator defined in \eqref{eq:A} with $\varSigma_{x,t}$ satisfying the geometric conditions described as above.
	Then there exists a constant $C$ depending on $p$ such that if $f\in L^p(\R^2)$, we have
	\begin{equation}
	\|Af\|_{L^p_{\gamma}(\R^{2+1})}\leqslant C\|f\|_{L^p(\R^2)},\quad 2\leqslant p\leqslant 6,
	\end{equation}
	for all $\gamma<-\bar{s}_p+\frac{1}{2}+\varepsilon(p)$.
\end{corollary}

\subsection{Square function estimates and an overview of the proof}
Let us briefly
recall the strategy in \cite{MSS-jams}, which reduced Theorem \ref{main} to a square function estimate for oscillatory integrals of H\"ormander type.
By interpolation with the trivial $L^2\to L^2$ estimate and the sharp $L^6\to L^6$ local smoothing estimate of \cite{BelHicSog18P},
\begin{equation}
\label{eq:L6}
\bigl\|\mathscr{ F }f\bigr\|_{L^6_{\rm loc}(Z)}\leqslant C\|f\|_{L^6_{\rm comp}(Y)}, \quad
\text{for all}\quad\sigma<-\frac{1}{6},
\end{equation}
one may reduce
\eqref{eq:main} to
\beq \label{eq:a1}
\bigl\|\mathscr{ F }f\bigr\|_{L^{\frac{10}{3}}_{\rm loc}(Z)}\leqslant C\|f\|_{L^{\frac{10}{3}}_{\rm comp}(Y)}, \quad
\text{for all}\quad\sigma<-\frac{1}{20}.
\eeq

Working microlocally, one may write
an operator
$\mathscr{F}$ in the class $I^{\sigma-1/4}(Z,Y;\mathscr C)$
with $\mathscr C$ satisfying the cinematic curvature condition in an appropriate local
coordinates as an
oscillatory integrals
\[
\mathscr{F}f(z)=\int e^{i\phi(z,\eta)}a(z,\eta) \widehat{f}(\eta)\,\dif\eta,
\]
where $a$ is a smooth symbol of order $\sigma$ supported in a specific conic subset  of $T^*Z$ and the cinematic curvature condition is manifested by explicit analytic properties of phase function $\phi$, which is homogeneous one in $\eta$. In particular, we may assume that $a(z,\eta)$ vanishes unless $\eta-$is contained in a small neighborhood of the north pole
$\mathbf{e}_2=(0,1)$.
By the standard Littlewood-Paley decomposition and scaling argument,
we are reduced to the study of an operator
\begin{equation}
\label{eq:T}
T_\lambda(f)(z):=\int
e^{i\lambda\phi(z,\eta)}
a_{\lambda}(z,\eta)\,{f}(\eta)\,
\dif\eta,
\end{equation}
with $\lambda$ being taken sufficiently large and $a_\lambda$ being a smooth symbol of order zero after an appropriate normalization\footnote{
We refer to the next section for more details}.
Next, one would like to introduce an additional decomposition with respect to the $\eta$-variable in the angular direction. Specifically, we make a covering over the unit circle of the $\eta$-variable  by sectors of radius   $\approx\lambda^{-1/2}$.
There is an associated decomposition of the oscillatory integral operator with respect to the angular directions
\beq
\label{eq:Tnu}
T_\lambda^\nu(f)(z)
=\int
e^{i\lambda\phi(z,\eta)}
a_{\lambda}^{\nu}(z,\eta){f}(\eta)\,\dif \eta,
\eeq
so that $T_\lambda (f)=\sum_\nu T^\nu_\lambda(f)$.
We shall return to this issue later with more details in the next section.

With the above preparation at hand,
there are two crucial ingredients towards our main theorem  in spirit of Mockenhaupt, Seeger and Sogge \cite{MSS-Ann,MSS-jams}.
The first one  is a sharp $L^2-$estimate for  variable coefficient versions of the Kakeya maximal function,
which had appeared in the work of C\'ordoba \cite{cor}on Bochner-Riesz multiplier problems. Let $Z$ and $Y$ be as in the last subsection and for $(y,\xi)\in\varPi_{T^{*}Y}(\mathscr C)$, we set
\begin{equation}
\label{eq:curves}
\gamma_{y,\xi}=\{z\in Z:(z,\zeta,y,\xi)\in\mathscr C, \;\text{for some } \zeta\},
\end{equation}
which is a smooth curve immersed in $Z$.
Fix a smooth metric on $Z$ and define for $\delta>0$ being a small number
\[
R^\delta_{y,\,\xi}=\{z\in Z:{\rm dist}(z,\gamma_{y,\xi})<\delta\}.
\]
Let $\alpha\in C_0^\infty(Y\times Z)$ and put
\[
\mathcal{M}_\delta g(y)=\sup_{\xi\in\varPi_{T^{*}_yY}(\mathscr C)}\frac{1}{{\rm Vol}(R^\delta_{y,\,\xi})}\Bigl|\int_{R^\delta_{y,\,\xi}}\alpha(y,z)g(z)\dif z\Bigr|.
\]
Then we have
\begin{equation}
\label{eq:kky-p-2}
\|\mathcal{M}_\delta\|_{L^p\to L^p}\leqslant C\Bigl(\log\frac{1}{\delta}\Bigr)^{\frac{1}{2}},\;\text{for } 2\leqslant p\leqslant\infty,
\end{equation}
 which is readily deduced from
the following theorem established in \cite{MSS-jams} by interpolating  with the trivial $L^\infty\to L^\infty$ estimate.
\begin{theorem}
	\label{kky}
	Assume $Z$ and $Y$ are para-compact smooth manifold equipped with a smooth metric and ${\rm dim}\, Z=3$, ${\rm dim}\, Y=2$. Suppose that $\mathscr C$ satisfies the cinematic curvature condition. Then there exists
	a constant $C$ such that if $g\in L^2(Z)$, then
	\begin{equation}
	\label{eq:kky}
	\|\mathcal{M}_\delta g\|_{L^2(Y)}\leqslant C\,\sqrt{\log\frac{1}{\delta}}\;\|g\|_{L^2(Z)}.
	\end{equation}
\end{theorem}

Aside from \eqref{eq:kky-p-2}, the second ingredient is to show the following square function estimate in order to obtain
Theorem \ref{main}. One may consult \cite{MSS-jams} for more details how our main theorem can be deduced from
\eqref{eq:kky-p-2} and \eqref{eq:sf} below.
\begin{proposition}
	\label{prop:sq-f}
	Let $Z$ and $Y$ be  smooth paracompact manifolds of dimension
	three and two respectively.
	Suppose that $\mathscr F\in I^{\sigma-\frac{1}{4}}(Z,Y;\mathscr C)$,
	is a Fourier integral operator of order $\sigma$, whose canonical relation $\mathscr C$ satisfies the cinematic curvature condition.
	Let $T_\lambda$ and $T^\nu_\lambda$ be
	given by
	\eqref{eq:T} and \eqref{eq:Tnu},
	 then for $\lambda$ being sufficiently large, one has  up to an error term which behaves like $\mathcal{O}(\lambda^{-N})$ for arbitrarily large $N$
	\begin{equation}
	\label{eq:sf}
	\bigl\|T_\lambda (f)\bigr\|_{L^{\frac{10}{3}}(\R^{3})}\lessapprox_{\phi,N,\varepsilon}\lambda^{\frac{1}{20}}\Bigl\|\bigl(\sum_\nu|{T}_\lambda^\nu (f)|^2\bigr)^{1/2}\Bigr\|_{L^{\frac{10}{3}}(\R^{3})}.
	\end{equation}
\end{proposition}

The proof Proposition \ref{prop:sq-f} will occupy the rest part of this  paper and will be  divided roughly into three parts.
We first explore an equivalence of \eqref{eq:sf} with its bilinear version in Section \ref{sect3} and then we dedicated Section \ref{sect4} to the proof of this bilinear square-function estimate based on a bilinear oscillatory integral estimates (see Section \ref{sect2}). The implementation of
bilinear oscillatory integral estimates calls for an application of the locally constant property which is allowed by the uncertainty principle. To this end, we need to introduce an additional decomposition along the radial directions cutting each sector into a union of blocks. Finally, in Section \ref{sect5}, we add up these blocks along the radial direction after the use of bilinear oscillatory integral estimates,  by adapting  a strategy of \cite{BelHicSog18P},  which is an approximation argument via stability property of oscillatory integrals estimates in square function norms.

We end up this section with an explanation on the reason why we only focus on the two dimensional case.
Indeed, the square function estimate
\eqref{eq:sf} can be generalized  verbatim to higher dimensions with ${\rm dim} \,Z=n+1$
and ${\rm dim}\, Y=n\geqslant 3$ as can be seen in the last section.
However, the intention of using the strategy of this paper to answer the question raised by Beltran, Hickman and Sogge in \cite{BelHicSog18P} on the local smoothing property of Fourier integral operators satisfying the cinematic curvature conditions when $p\leqslant q_n$ with or without the convexity assumption, is less satisfactory  when $n\geqslant 3$.
Indeed, one may check that the regularity result of higher dimensions derived by means of the method used in this paper is even worse than that obtained simply by interpolating the
sharp  $L^{q_n}\to L^{q_n}$ estimate in
\cite{BelHicSog18P} and the trivial $L^2\to L^2$ estimate.
This is because the bilinear oscillatory integral estimates as well as $L^p-$ estimate for Kakeya maximal functions would have to afford an amount of loss of derivatives in dimensions higher than two.

 However, this fact does not mean that the Mockenhaupt-Seeger-Sogge
 approach via ``
 square functions $\oplus$ Kakeya"
 is not promising towards the resolution of the local smoothing conjecture, at least in the constant variable setting.

\subsection*{Notations}
If $a$ and $b$ are two positive quantities, we write $a\lesssim b$ when there exists a constant $C>0$ such that $a\leq C b$ where the constant will be clear from the context. When the constant depends on some other quantity $M$, we emphasize the dependence by writing $a\lesssim_M b$.
We will write $a\approx b$ when we have both $a\lesssim b$ and $b\lesssim a$. We will write $a\ll b$ (resp. $a\gg b$) if there exists a sufficiently  large constant
$C>0$ such that $Ca\leq  b$ (resp. $a\geq Cb$).
We adopt the notion of nature numbers $\mathbb{N}=\mathbb{Z}\cap [0,+\infty)$.
For $\lambda\gg 1$,
we use ${\rm RapDec}(\lambda)$ to mean a quantity rapidly decreasing in $\lambda$.
We use $a\lessapprox b$
to mean $a\lesssim_\varepsilon \lambda^\varepsilon b$ for arbitrate $\varepsilon$.
\\
\\
Throughout this paper, $w_B$ is a rapidly decaying weights concentrated on the ball centered at $c(B)$ and of radius $r(B)$,
\[
w_B(z)\lesssim \left(1+\frac{|z-c(B)|}{r(B)}\right)^{-N},\quad N\gg 1.
\]

\section{Preliminaries and Reductions}\label{sect2}
Given a point $(z_0,\zeta_0,y_0,\eta_0)\in T^*Z\setminus 0\times T^*Y\setminus 0$,
there exists a sufficiently small local conic coordinate patch around it,
along with a smooth function $\phi(z,\eta)$
such that $\mathscr C$ is given by
\begin{equation}
\label{qmmkds}
\{
(z,\phi'_z(z,\eta),\phi'_\eta(z,\eta),\eta):
\eta \in(\R^2\setminus 0)\cap \varGamma_{\eta_0}
\}
\end{equation}
where $\varGamma_{\eta_0} $ denotes a conic neighborhood of $\eta_0$.

By splitting
$z=(x,t)\in\R^2\times\R$ into space-time variables, where
we put $z_0=\mathbf{0}$ without loss of generality,
any operator
$\mathscr{F}$ in the class $I^{\sigma-1/4}(Z,Y;\mathscr C)$
with $\mathscr C$ satisfying the cinematic curvature condition can be written in an appropriate local
coordinates as a finite sum of
oscillatory integrals
\[
\mathscr{F}f(x,t)=\int_{\R^n} e^{i\phi(x,t,\eta)}b(x,t,\eta) \widehat{f}(\eta)\,\dif\eta,
\]
where $b$ is a smooth symbol of order $\sigma$.  We may assume that the support of the  map $z\to b(z,\eta) $ is contained in a ball
$B(\mathbf{0},\varepsilon_0)$, with $\varepsilon_0>0$ being sufficiently small and $\eta\to b(z,\eta)$ is supported in a conic region $\mathcal{V}_{\varepsilon_0}$, i.e.
$$b(x,t,\eta)=0
\;\text{if}\; \eta\notin\mathcal{V}_{\varepsilon_0}
:=\{\xi=(\xi_1,\xi_2)\in\R^2\backslash 0:|\xi_1|\leqslant \varepsilon_0\, \xi_2\}.$$

Fix $\lambda\gg 1$ and $\beta\in C_c^\infty(\mathbb{R})$ which vanishes outside  the interval $(1/4, 2 )$  and equals one in $(1/2, 1)$. By standard Littlewood-Paley decomposition,
\eqref{eq:a1} can be deduced from
\beq \label{eq:34}
\|\mathscr{F}_{\lambda}f\|_{L^{\frac{10}{3}}(\R^3)}
\lessapprox \lambda^{\frac{1}{20}}\|f\|_{L^{\frac{10}{3}}{(\R^2})},
\eeq
where
$\mathscr{F}_\lambda $ is an operator
\begin{equation}
\label{mklsmcd}
\mathscr{F}_\lambda f(x,t)=
\int_{\R^2}
e^{i\phi(x,t,\eta)}
b^\lambda (x,t,\eta)\widehat{f}(\eta)\,\dif\eta,
\end{equation}
and
\beq
b^\lambda(z,\eta)=b(z,\eta)\frac{1}{(1+|\eta|^2)^{\sigma/2}}\beta\Big(\frac{\eta}{\lambda}\Big).
\eeq

Let $a(z,  \eta)\in C_c^\infty(\R^3\times \R^2)$ with compact support contained in $B(0,\varepsilon_0)\times B({\bf e_2},\varepsilon_0)$. Assume $\mathcal{C}(\mathbf{e}_2,\varepsilon_0)
:=B(\mathbf{e}_2,\varepsilon_0)\cap\mathbb{S}^{1}$ and
make angular decomposition with respect to the $\eta$-variable by cutting
$\mathcal{C}(\mathbf{e}_2,\varepsilon_0)$ into
$N_\lambda\approx\lambda^{\frac{1}{2}}$ many sectors
$\{\theta_{\nu}: 1 \leqslant \nu \leqslant N_ \lambda\}$,
each  $\theta_{\nu}$  spreading an angle  $\approx_{\varepsilon_0}\lambda^{-1/2}$.
We denote by $\kappa_\nu\in \mathbb{S}^{1}$  the center of $\theta_{\nu}$.

Let $\{\chi_{\nu}(\eta)\}$ be a series of smooth  cutoff function associated with the decomposition in the angular direction,
each of which is  homogeneous of degree $0$,
such that $\{\chi_\nu\}_{\nu}$ forms a partition  of unity on the unit circle
and then extended homogeneously to $\R^2\setminus 0$ such that
\begin{equation*}\left\{\begin{aligned}
&\sum_{0\leqslant \nu\leqslant N_\lambda} \chi_{\nu}(\eta)\equiv 1,\;\;\forall \eta \in \mathbb{R}^2\setminus 0,\\
&|\partial^{\alpha} \chi_{\nu}(\eta)|\leqslant C_\alpha \lambda^{\frac{|\alpha|}{2}},\;\; \forall \;\alpha\in \mathbb{N}^2 \; \text{if}\; |\eta|=1.\end{aligned}\right.\end{equation*}
Define
\begin{align}
\label{eq:TT}
T_\lambda f(z)=\int e^{i\lambda \phi(z,\eta)}a(z,\eta) f(\eta)\,\dif\eta,\;\;
T_\lambda^\nu f(z)
=\int
e^{i\lambda\phi(z,\eta)}
a^{\nu}(z,\eta)f(\eta)\,\dif \eta,
\end{align}
where $a^{ \nu}(z,\eta)=\chi_\nu(\eta)a(z,\eta)$.
Then $T_\lambda f=\sum_\nu T_\lambda^\nu f$.
As explained in the first section, we would have \eqref{eq:a1} provided we could have proved the following square function estimate
\beq \label{eq:333'}
\|T_{\lambda}f\|_{
	L^{\frac{10}{3}}(\R^{2+1})}
\lessapprox \lambda^{1/20}
\Bigl\|\bigl(\sum_{\nu}
|T_{\lambda}^{\nu}f|^2\bigr)^{1/2}\Bigr\|_{L^{\frac{10}{3}}(\R^{2+1})}+{\rm RapDec}(\lambda)\|f\|_{L^{\frac{10}{3}}(\R^2)}.
\eeq

For technical reasons, we  assume $a$ is of the form  $a(z,\eta)=a_1(z) a_2(\eta)$, where $$a_1\in C_c^\infty (B(0,\varepsilon_0)), \quad a_2\in C_c^\infty(B({\bf e_2},\varepsilon_0)).$$ The general cases may be reduced to this special one via  the following observation
\begin{align*}
T_\lambda f(z)=\int_{\R^3} e^{i(z,\xi)}
\Bigl(\int_{\R^2} e^{i\phi(z,\eta)}\psi(z)\widehat{a}(\xi, \eta)  f(\eta) \dif \eta\Bigr) \dif \xi,
\end{align*}
and that $\xi\mapsto\widehat{a}(\xi,\eta)$ is a Schwartz function,
where $\psi(z)$ is a compactly supported smooth function and equals $1$ on ${\rm supp}_z\; a$.

We may reformulate
\eqref{dkmkldsc} \eqref{vmklsdmks} and the curvature condition  as
\begin{itemize}
	\item[$\mathbf{H}_1$] ${\rm rank} \;\partial_{z\eta}^2 \phi(z,\eta)=2$  for all $(z,\eta)\in  {\rm supp}\;a$.
	\item[$\mathbf{H}_2$]  Define the Gauss map $G: {\rm supp}\;a \rightarrow \mathbb{S}^{2}$ by $G(z,\eta):=\frac{G_0(z,\eta)}{|G_0(z,\eta)|}$ where
	\beq
	G_0(x,\eta):=\partial_{\eta_1}\partial_z\phi(z,\eta)\wedge \partial_{\eta_2}\partial_z\phi(z,\eta).
	\eeq
	The curvature condition
	\beq
	{\rm rank}\; \partial_{\eta\eta}^2\langle \partial_z\phi(z,\eta), G(z,\eta_0) \rangle|_{\eta=\eta_0}=1
	\eeq
	holds for all $(z,\eta_0)\in {\rm supp}\;a$.
\end{itemize}
The process of reduction in the above paragraphs is quite standard and we recommend the reader to consult the works \cite{BelHicSog18P-surv,BelHicSog18P,MSS-Ann,MSS-jams}  in the literature.

\subsection{Normalization of the phase function}

The conditions  $\mathbf{H}_1$, $\mathbf{H}_2$
imply  that there exists a special coordinate system,
so that the phase function $\phi(z,\eta)$ can be written in a \emph{normalized} form.
More precisely,
if we write $z=(x,t)$ and assume the normal vector $G(z,\eta)$ is parallel to the $t$-direction at $(\mathbf{0},\mathbf{e}_2)$,
then up to multiplying harmless factors to $T_\lambda f$ and $f$,
we can write $\phi$ in this coordinate as
\beq\label{eq:18}
\phi(x,t,\eta)=\langle x,\eta\rangle
+\frac{t}{2}\,
\partial_t\partial_{\eta_1}^2\phi(\mathbf{0},\mathbf{e}_2)\eta_1^2/\eta_2
+\eta_2\,\mathcal{E}(x,t,\eta_1/\eta_2)
\eeq
where $\eta=(\eta_1,\eta_2) $ and  $\mathcal{E}(x,t,s)$ obeys
\beq\label{eq:14}
\mathcal{E}(x,t,s)=O\bigl((|x|+|t|)^2s^2+(|x|+|t|)|s|^3\bigr).
\eeq
An additional change of variables in $t$
allows us to assume $\partial_t\partial_{\eta_1}^2\phi(\mathbf{0},\mathbf{e}_2)=1$.
Furthermore, for any given $N$ being sufficiently large, we may also assume the uniform bound of the higher order derivatives of the phase functions, analytically,
\beq \label{eq:191}
|\partial_\eta^\beta \partial_z^\alpha\phi(z,\eta)|\leqslant 1/2, \; 0\leqslant |\beta|\leqslant N, |\alpha|=2,\; \text{for all}\; (z,\eta) \in {\rm supp}\;a.
\eeq
Otherwise, \eqref{eq:191} can be guaranteed by replacing $\phi(z,\eta)$ by $A\phi(z/A,\eta)$ with $A>0$ being sufficiently large depending on $N$ and $\phi$ along with its derivatives evaluated on the support of $a$. For more details, see Beltran, Hickman, Sogge \cite{BelHicSog18P}, Lee \cite{Lee-JFA}, and the previous works of Bourgain \cite{Bo} and H\"ormander \cite{Ho}.

\vskip0.2cm
 Let $\Upsilon_{x,t}:\eta\rightarrow \partial_{x}\phi(x,t,\eta)$.
If $\varepsilon_0$ is taken sufficiently small,
$\Upsilon_{x,t}$ is a local diffeomorphism on $B(\mathbf{e}_2,\varepsilon_0)$.
If we denote by $\Psi_{x,t}(\xi)=\Upsilon_{x,t}^{-1}(\xi)$  the inverse map of $\Upsilon_{x,t}$,
then clearly
\beq \label{eq:44}
\partial_x \phi(x,t, \Psi_{x,t}(\xi))=\xi.
\eeq
Differentiating  \eqref{eq:44} with respect to $\xi$ on both sides yields
\beq \label{eq:50}
[\partial^2_{x,\eta}\phi](x,t,\Psi_{x,t}(\xi))\;\partial_{\xi}
\Psi_{x,t}(\xi)={\rm Id}.
\eeq
This manifests that
\begin{equation}
\label{eq:mmmm}
{\rm det}\,\partial_{\xi}\Psi_{x,t}(\xi)\neq 0,\;\forall(x,t)\in B(\mathbf{0},\varepsilon_0),\;\forall \xi\in\Upsilon_{x,t}(B(\mathbf{e}_2,\varepsilon_0)).
\end{equation}
Let
\beq
\label{q} q(x,t,\xi)=\partial_t\phi(x,t,\Psi_{x,t}(\xi)).
\eeq
Then, we have
\beq\label{eq:10}
\partial_t \phi(x,t,\eta)
=q(x,t,\partial_x\phi(x,t,\eta)).
\eeq
\subsection{A bilinear estimate for oscillatory integrals }
For $j=1,2$, we consider two oscillatory integral operators
\beq
W_{\lambda}^j  f(z)=\int e^{i\lambda \phi_j (z,\eta)}a_j(z,\eta)f(\eta)\,\dif\eta,\quad z=(x,t)\in \R^{2}\times \R,
\eeq
where $\phi_j$ satisfies $\mathbf{H}_1,\mathbf{H}_2$ and $a_j\in C_c^\infty(\R^3\times\R^2)$.
We shall use
the following bilinear oscillatory integral estimates  established in \cite{Lee-JFA},
which is an extension to the variable coefficient case of the bilinear adjoint restriction theorem of Wolff \cite{Wolff}.
Specifically, for each multi-index $\beta \in \mathbb{N}^3$,  assuming that
\beq
|\partial_z^\beta \phi_j (z, \eta)|\leqslant A_\beta, \; \text{for}\; 0\leqslant |\beta|\leqslant N, j=1,2,\; \text{for all}\; (z,\eta) \in {\rm supp}\;a,
\eeq
we have the following theorem.
\begin{theorem}[\cite{Lee-JFA}]
	\label{theo2}
	Let $\phi_j$ with $j=1,2$ be smooth homogeneous functions of degree $1$ with respect to $\eta$,
	which satisfies conditions $\mathbf{H}_1$ $\mathbf{H}_2$.
	Assume that $\partial_x\phi_j\neq 0$ on the support of $a_j$,
	which is small enough such that $\phi_j$ satisfies \eqref{eq:10}.
	Suppose that
	\beq
	{\rm rank}\;\partial_{\eta\eta}^2q_j(x,t,\partial_x\phi_j(x,t,\eta^{(j)}))=1
	\eeq
	on ${\rm supp} \;a_j$ and
	\beq\label{eq:11}
	\Bigl|\Bigl\langle\frac{\partial_x\phi_j(z,\eta^{(j)})}{|\partial_x\phi_j(z,\eta^{(j)})|},
	\partial_\eta q_1\bigl(z,\partial_x\phi_1(z,\eta^{(1)})\bigr)-\partial_\eta q_2\bigl(z,\partial_x\phi_2(z,\eta^{(2)})\bigr)
	\Bigr\rangle\Bigr|\geqslant c_0>0
	\eeq
	for $j=1,2$, whenever $(z,\eta^{(1)})\in {\rm supp}\;a_1$ and $(z,\eta^{(2)})\in {\rm supp}\;a_2$.
	Then for every $ p\geqslant \frac{5}{3}$,
	one has the bilinear estimate
	\beq \label{eq:36}
	\|W_\lambda^1 f W_\lambda^2 g\|_{L^p(\R^{2+1})}\leqslant C(A_\beta, \varepsilon, c_0)\lambda^{-\frac{3}{p}+\varepsilon}\|f\|_{L^2(\R^2)}\|g\|_{L^2(\R^2)}.
	\eeq
	for a finite number of $\beta$.
\end{theorem}

It remains to prove \eqref{eq:333'}. To this end, we will employ in the following context the bilinear approach of \cite{Tao-Vargas-II,GaSe09} and  stability property  which enables us to approximate the oscillatory integral operator by extension operator at an appropriate small scale.

\section{ Reduction  to bilinear square-function estimate}\label{sect3}
We will convert \eqref{eq:333'} to its bilinear equivalent version which will be then  established based on bilinear oscillatory estimate. One direction of the implication is a straightforward application of the H\"older's inequality, the reverse direction is delicate.
\begin{proposition}
	\label{Pro3}
	Assume $\Omega, \Omega'$ are two sets consisting of $\nu$ and $\nu'$ respectively, which  satisfy  the following angular separation condition:
	\beq \label{eq:181}
	\operatorname{Ang}(\theta_\nu, \theta_{\nu'})\approx_{\varepsilon_0}1, \;\; \text{for every pair} \; (\nu, \nu')\in \Omega\times\Omega',
	\eeq
	where $\operatorname{Ang}(\theta_\nu, \theta_{\nu'})$ measures the angle between $\theta_\nu, \theta_{\nu'}$. Let
	$T^\nu_\lambda$ be the operator defined in \eqref{eq:T}, with $\phi$ satisfying $\mathbf{H}_1, \mathbf{H}_2$ and  $a\in C_c^\infty(\R^3\times\R^2)$.
	If
	\begin{align}
	&\Bigl\|\sum_{\nu \in \Omega} T_{\lambda}^{\nu}g\sum_{\nu'\in \Omega'} T_{\lambda}^{\nu'}h\Bigr\|_{L^{5/3}(\R^{2+1})}\nonumber\\
	\lessapprox& \lambda^{1/10}\Bigl\|\Big(\sum_{\nu\in \Omega}|T_{\lambda}^{\nu}g|^2\Big)^{\frac12}\Bigr\|_{L^{10/3}(\R^{2+1})}
\Bigl\|\Big(\sum_{\nu'\in \Omega}|T_{\lambda}^{\nu' }h|^2\Big)^{\frac12}\Bigr\|_{L^{10/3} (\R^{2+1})},\label{eq:w5d}
	\end{align}
	holds, up to a ${\rm RapDec}(\lambda)$ term, for all functions $g, h$, where the  implicit constant depends on the constant $c_0$ appearing in \eqref{eq:11} and $A_\beta$ for finite many $\beta$'s, then  we have \eqref{eq:333'}.
\end{proposition}
\begin{proof}
	We first perform a Whitney type decomposition of the product sectors with respect to dyadic scales between $\lambda^{-1/2}$
	and $1$,  and then single out the contribution of bilinear forms  corresponding to  pairs of the generated  subsectors
	for the off-diagonal part at each individual dyadic level by an orthogonality argument. In the second step, we use parabolic rescaling to reduce  the bilinear forms at each dyadic scale to the  situation in the statement of the proposition.
	\subsubsection*{Step 1. Orthogonality argument  }
	Let $j_0$ be the largest integer with
	$2^{j_0}\leqslant \lambda^{1/2}$.
	For each $j\in\Z$ with $|\log_2\varepsilon_0|\leqslant j\leqslant j_0$,
	we denote by $\{\theta_{j,l}\}_l$ a collection of sectors  of scale  $\approx2^{-j}$,
	where $l\in\{1,2,\ldots,l_j\}$
	for some $l_j\approx 2^j $, with the property  that for each $l$,
	$\theta_{j,l}$ is a union of $\approx2^{-j}\lambda^{1/2}$
	many of consecutive $\theta_{\nu}$'s as introduced in
	Section \ref{sect2}.
	Take a Whitney-type decomposition of
	$\mathcal{C}({\mathbf{e}_2},\varepsilon_0)\times \mathcal{C}({\mathbf{e}_2},\varepsilon_0)$
	\begin{multline*}
	\mathcal{C}({\mathbf{e}_2}, \varepsilon_0)\times \mathcal{C}({\mathbf{e}_2},
	\varepsilon_0)
	\\
	=\Bigl(\bigcup_{j:\varepsilon_0^{-1}\leqslant 2^j\leqslant \varepsilon_0^3 2^{j_0}}
	\bigcup_{{\rm Ang}(\theta_{j,l}, \theta_{j,{l'}})\approx 2^{-j} }\theta_{j,l}\times \theta_{j,{l'}}\Bigr)\cup
	\Bigl(\bigcup_{{\rm Ang}(\theta_{{j_0},l}, \theta_{{j_0},{l'}})
		\leqslant \varepsilon_0^{-3} 2^{-j_0} }\theta_{{j_0},l}\times \theta_{{j_0},{l'}}\Bigr),
	\end{multline*}
	where ${\rm  Ang  }(\theta_{j,l}, \theta_{j,{l'}})\approx 2^{-j}$ means
	$2^{-(j+1)}\leqslant{\rm  Ang}(\theta_{j,l}, \theta_{j,{l'}})\leqslant 2^{-(j-1)}$.
	For $1\leqslant l \leqslant l_j$, we put
	$\chi_{j,l}(\eta)=\sum\limits_{\nu:\;\theta_{\nu}\subset \theta_{j,l}}\chi_\nu(\eta),$
	and  set
	$a^{j,l}(x,t,\eta)=\chi_{j,l}(\eta)a(x,t,\eta).$
	Define
	\beq
	\label{klkl}
	(T_{\lambda}^{j,l}f)(x,t)=\int e^{i \lambda\phi(x,t,\eta)}{a}^{j,l}(x,t,\eta) f(\eta)\,\dif\eta.
	\eeq
	Correspondingly, we have
	\begin{align}
	\label{kkkddkk}
	\left((T_{\lambda}f)(x,t)\right)^2
	=&\sum_{(l,l')\in\Lambda_{j_0}}(T_{\lambda}^{j_0,l}f)(x,t)(T_{\lambda}^{j_0,l'}f)(x,t)\\
	&+\sum_{j: \; \varepsilon_0^{-1}\leqslant 2^j\leqslant\varepsilon_0^3 2^{j_0}}
	\sum_{(l,l')\in\Lambda_j}(T_{\lambda}^{j,l}f)(x,t)(T_{\lambda}^{j,l'}f)(x,t),
	\label{eq:off-diag}
	\end{align}
	where $\Lambda_j$ denotes the collection of $(l,l')$
	indicating the sectors satisfying
	${\rm Ang}(\theta_{j,l}, \theta_{j,{l'}})\approx 2^{-j}$ for $\varepsilon_0^{-1}\leqslant 2^j\leqslant \varepsilon_0^32^{j_0}$,
	and
	$\Lambda_{j_0}$ denotes the collection of $(l,l')$
	indexing the sectors satisfying
	${\rm Ang}(\theta_{{j_0},l}, \theta_{{j_0},{l'}})\leqslant \varepsilon_0^{-3} 2^{-j_0}$.
	Therefore, by Minkowski's inequality
	\begin{align}
	\label{aakkkddkk}
	\|T_\lambda f\|^2_{L^{10/3}_{x,t}}=
	\bigl\|(T_{\lambda}f)^2\bigr\|_{L^{5/3}_{x,t}}
	\leqslant&\Bigl\|\sum_{(l,l')\in\Lambda_{j_0}}(T_{\lambda}^{j_0,l}f)(T_{\lambda}^{j_0,l'}f)\Bigr\|_{L^{5/3}_{x,t}}\\
	&+\sum_{j: \;\varepsilon_0^{-1}\leqslant 2^{j}\leqslant \varepsilon_0^32^{j_0}}\Bigl\| \sum_{(l,l')\in\Lambda_j}(T_{\lambda}^{j,l}f)(T_{\lambda}^{j,l'}f)\Bigr\|_{L^{5/3}_{x,t}}.\label{eq:off}
	\end{align}	
	
	For the first term  \eqref{aakkkddkk},  since each $\theta_{j_0,l}$ involves finitely many $\theta_\nu$'s, by Schur's test we have
	\[
	\Bigl\|\sum_{(l,l')\in\Lambda_{j_0}}(T_{\lambda}^{j_0,l}f)(T_{\lambda}^{j_0,l'}f)\Bigr\|_{L^{5/3}(\R_{x,t}^{2+1})}
	\lesssim \Bigl\|\bigl(\sum_{\nu}|T_{\lambda}^{\nu}(f)|^2\bigr)^{1/2}\Bigr\|_{L^{10/3}(\R^{2+1}_{x,t})}^2.
	\]
	
	For \eqref{eq:off},  we will use the assumption \eqref{eq:w5d}.
	To this end,  we will exploit the orthogonality property to prove, up to a ${\rm RapDec}(\lambda)$ term
	\begin{equation}
	\label{aaldsd}
	\Bigl\| \sum_{(l,l')\in\Lambda_j}(T_{\lambda}^{j,l}f)(T_{\lambda}^{j,l'}f)\Bigr\|_{L^{5/3}(\R^{2+1}_{x,t})}
	\lessapprox
	\Big(
	\sum_{(l,l')\in\Lambda_j}
	\Bigl\| (T_{\lambda}^{j,l}f)(T_{\lambda}^{j,l'}f)
	\Bigr\|^{5/3}_{L^{5/3}(\R^{2+1}_{x,t})}
	\Big)^{3/5},
	\end{equation}

	The proof is essentially from \cite{Lee-JFA}. For the sake of self-containedness, we sketch it below.
	Let $\psi(x)$ be a smooth function with its support contained in $B(0,\varepsilon_0) $ such that if we set $\psi^{\mu}_{j}(x):=\psi(2^jx-\mu)$, then we have
	\beq
	\sum_{\mu \in \varepsilon_0\mathbb{Z}^2 } \psi_{j}^{\mu}(x)\equiv 1.
	\eeq
	Set $\Phi(x,t,\eta,\eta')=
	\phi(x,t,\eta)+\phi(x,t,\eta')$ and define
	\begin{align*}
	&\mathcal{P}^{\mu,j}_{\lambda,l,l'}(\xi,t,\eta,\eta')=\int e^{i\lambda\Phi(x,t,\eta,\eta')-i x\cdot\xi}
	A_{l,l'}^{\lambda,j}(x,t,\eta,\eta')\psi^{\mu}_{j}(x)\,\dif x,\\
	&\mathcal{A}^{\mu,j}_{\lambda,l, l'}(f)(x,t)=
	\psi^{\mu}_{j}(x)\iint e^{i\lambda\Phi(x,t,\eta,\eta')}
	A_{l,l'}^{\lambda,j}(x,t,\eta,\eta')
	f (\eta)f(\eta')\,\dif\eta \dif\eta',
	\end{align*}
	where
$$A_{l,l'}^{\lambda,j}(x,t,\eta,\eta')=a^{j,l}_\lambda(x,t,\eta)a^{j,l'}_\lambda(x,t,\eta').$$
	Denoting by
	$r=|\eta|$, $r'=|\eta'|$ and by
	$\kappa_{j}^{l},\kappa_{j}^{l'}$ the center of $\theta_{j,l}$ and $\theta_{j,l'}$ respectively,
	in view of the support of $\psi^{\mu}_{j}(x)$ and the cone condition,
	we have
	$$	\partial_{x}(\lambda\Phi(x,t,\eta,\eta')-x\xi)=\lambda\partial_x\Phi( 2^{-j}\mu,t,r\kappa_{j}^l,r'\kappa_{j}^{l'})-\xi+O(\lambda2^{-j}). $$
		It follows from integration by parts that for
	$\angle(\eta,\kappa_{j}^{l})\leqslant \varepsilon_0 2^{-j}$ and $\angle(\eta',\kappa_{j}^{l'})\leqslant \varepsilon_0 2^{-j}$
	with $(l,l')\in\Lambda_j$,
	the function $\xi\to \mathcal{P}^{\mu,j}_{\lambda,l,l'}(\xi,t,\eta,\eta')$ decreases fast outside the rectangle
	$$
	\mathcal{R}_{\lambda,l}^{\mu,j}(t):=
	\Big\{\xi\in \R^2:|\xi-\lambda r\partial_x\phi(2^{-j}\mu,t,\kappa_{j}^{l})|\lessapprox\lambda2^{-j}, r\in (2-2\varepsilon_0, 2+2\varepsilon_0)\Big\}.
	$$
	We define a smooth bump function $R^{\mu,j}_{\lambda,l}(\xi,t)$ associated  with $\lambda^{\varepsilon}\mathcal{R}^{\mu,j}_{\lambda,l}(t)$,
	then, we have
	\beq\label{eq:37}
	\Bigl\|({\rm Id}-R^{\mu,j}_{\lambda,l}(D_x,t))\mathcal{A}_{\lambda,l,l'}^{\mu,j}(f)(t,\cdot)\Bigr\|_{L^{10/3}(\R^2_x)}\lesssim {\rm RapDec}(\lambda)\|f\|_{L^{10/3}(\R^2)}^2.
	\eeq
	Since the $\xi$-support of $ R^{\mu,j}_{\lambda,l} (\xi,t)$ is contained in $\lambda^{\varepsilon}\mathcal{R}^{\mu,j}_{\lambda,l}$, by the nondegeneracy condition,  we know  for fixed  $\lambda^{\varepsilon}\mathcal{R}^{\mu,j}_{\lambda,l}$, there are at most $\approx\lambda^{\varepsilon}$ such sets  intersecting  with it with nonempty content.
	By using this finite-overlapping  property  we may conclude up to a fast decay term
	\begin{equation}
	\label{aaldsd2}
	\Bigl\| \sum_{(l,l')\in\Lambda_j}(T_{\lambda}^{j,l}f)(T_{\lambda}^{j,l'}f)\Bigr\|_{L^{\frac53}(\R^{2+1}_{x,t})}
	\lessapprox
	\Big(\sum_{(l,l')\in\Lambda_j}
	\Bigl\| (T_{\lambda}^{j,l}f)(T_{\lambda}^{j,l'}f)
	\Bigr\|^{\frac53}_{L^{\frac53}(\R^{2+1}_{x,t})}
	\Big)^{\frac35}.
	\end{equation}
	
	Going back to \eqref{aaldsd},
	we are reduced  to show, up to a negligible term,
	\begin{align}
		&\Bigl\|
	(T_{\lambda}^{j,l}f)
	(T_{\lambda}^{j,l'}f)
	\Bigr\|_{L^{5/3}(\R_{x,t}^{2+1})}\nonumber\\
	\lessapprox& (\lambda 2^{-2j})^{\frac{1}{10}}
	\Bigl\|
	\Bigl(\sum_{\nu:\theta_\nu \subset \theta_{j,l}}
	|T_{\lambda}^{\nu}(f)|^2
	\Bigr)^{\frac12}
	\Bigr\|_{L^{\frac{10}{3}}(\R_{x,t}^{2+1})}
	\Bigl\|
	\Bigl(\sum_{\nu':\theta_{\nu'}\subset  \theta_{j,l'}}
	|T_{\lambda}^{\nu'}(f)|^2\Bigr)^{\frac12}
	\Bigr\|_{L^{\frac{10}{3}}(\R^{2+1}_{x,t})}.
	\label{aamkdsc}
	\end{align}
	Indeed, substituting \eqref{aamkdsc}
	to \eqref{aaldsd}, we obtain, up to a ${\rm RapDec}(\lambda)$ term, by using Cauchy-Schwarz's inequality and $\ell^{2}\hookrightarrow\ell^{10/3}$
	\[
	\Bigl\| \sum_{(l,l')\in\Lambda_j}
	(T_{\lambda}^{j,l}f)
	(T_{\lambda}^{j,l'}f)
	\Bigr\|_{L^{\frac53}(\R_{x,t}^{2+1})}
	\lessapprox (\lambda2^{-2j})^{\frac1{10}}
	\Bigl\|\Big(\sum_{\nu}|T_{\lambda}^{\nu}(f)|^2\Big)^{\frac12}\Bigr\|_{L^{\frac{10}3}(\R_{x,t}^{2+1})}^2.
	\]
	Summing over $j$, we obtain \eqref{eq:333'}.

	\subsubsection*{ Step 2. Parabolic rescaling }
	In this step, we derive \eqref{aamkdsc} from  \eqref{eq:w5d}, where the angular separation condition can be verified via parabolic rescaling.
	
	Let $\eta^{j,l}$ be the center of $\theta_{j,l}$ and set $\alpha_{j}^l=\eta^{j,l}_1/\eta^{j,l}_2$.
	Then for every  $\eta \in \theta_{j,l}$, we have $$\left|\frac{\eta_1}{\eta_2}-\alpha_{j}^l \right|\leqslant 2^{-j}.$$
	By change of variable $\eta_1\rightarrow \alpha_j^l  \eta_2+\eta_1$, we may rotate the center of $\theta_{j,l}$ to the   $\mathbf{e}_2$ axis.  Correspondingly, we will perform the change of variable in the space variable $x$.
	\begin{equation}
	\label{eq:Q}
	\left\{\begin{aligned}
	x_1+\alpha_{j}^lt &=x_1^{(1)}\\
	\alpha_j^l x_1+x_2+\frac{1}{2}(\alpha_{j}^l)^2 t&=x_2^{(1)}\\
	t&= t^{(1)}
	\end{aligned}\right.
	\end{equation}
	For convenience, we will use $x^{(1)}$ to denote $x^{(1)}:=(x_1^{(1)}, x_2^{(1)})$. Since  the transformation above is a diffeomorphism, we use
	$\Phi^{(1)}(x^{(1)}, t^{(1)})$ to denote the inverse map of \eqref{eq:Q}.
	Under  the new variable system, the phase function $\phi$ is transformed to
	\beq \label{eq:121}
	\phi^{(1)}(x^{(1)}, t^{(1)}, \eta)=\langle x^{(1)},\eta\rangle+\frac{1}{2}t^{(1)} \frac{\eta_1^2}{\eta_2}
+\eta_2\mathcal{E}_1\Big( x^{(1)}, t^{(1)}, \frac{\eta_1}{\eta_2}+\alpha_{j}^l\Big)
	\eeq
	where
$$ \mathcal{E}_1\Big(x^{(1)}, t^{(1)}, \frac{\eta_1}{\eta_2}+\alpha_{j}^l\Big)=\mathcal{E}\Big( \Phi^{(1)}(x^{(1)}, t^{(1)}), \frac{\eta_1}{\eta_2}+\alpha_{j}^l\Big).$$
	Next, making  Taylor's expansion of $\mathcal{E}_1(x^{(1)},t^{(1)},  \eta_1/\eta_2+\alpha_{j}^l )$ as follows
	\begin{align*}
	\mathcal{E}_1(x^{(1)}, t^{(1)} ,  \eta_1/\eta_2+\alpha_{j}^l)&=\mathcal{E}_{1}(x^{(1)},t^{(1)} , \alpha_{j}^l )+\partial_s\mathcal{E}_{1}(x^{(1)},t^{(1)} , \alpha_{j}^l)\frac{\eta_1}{\eta_2}\\
	&+\frac{1}{2}\partial_s^2\mathcal{E}_1(x^{(1)},t^{(1)},  \alpha_{j}^l )\big(\frac{\eta_1}{\eta_2}\big)^2\\
	&+\frac{1}{2}\int_0^1 \partial_s^3\mathcal{E}_{1}(x^{(1)},t^{(1)}, s\frac{\eta_1}{\eta_2}+\alpha_{j}^l )\big(\frac{\eta_1}{\eta_2}\big)^3(1-s)^2 ds
	\end{align*}
	Making the change of variables
	\begin{equation}
	\label{eq:122}
	\left\{
	\begin{aligned}
	x_1^{(1)}+\partial_{s}\mathcal{E}_1(x^{(1)}, t^{(1)},   \alpha_{j}^l)&=x_1^{(2)}\\
	x_2^{(1)}+\mathcal{E}_1(x^{(1)}, t^{(1)}, \alpha_{j}^l)&=x_2^{(2)}\\
	\frac{1}{2}t^{(1)}+\frac{1}{2}\partial_{s}^2\mathcal{E}_{1}(x^{(1)}, t^{(1)}, \alpha_{j}^l )&= \frac{1}{2} t^{(2)}
	\end{aligned}\right.
	\end{equation}
	where as above, we use   $\Phi^{(2)}(x^{(2)}, t^{(2)})$ to denote the inverse map of \eqref{eq:122}, the phase function of \eqref{eq:121} is changed to
	\begin{align}\label{eq:19}
	\phi^{(2)}(x^{(2)}, t^{(2)}, \eta)&=\langle x^{(2)},\eta\rangle+\frac{1}{2} t^{(2)}\eta_1^2/\eta_2+\mathcal{E}_{2}\big(x^{(2)}, t^{(2)},\eta_1/\eta_2\big),
	\end{align}
	where
	\beq \label{eq:171}
	\mathcal{E}_2\Bigl(x^{(2)}, t^{(2)}, \frac{\eta_1}{\eta_2}\Bigr)=\frac{1}{2}\int_0^1 \partial_s^3 \mathcal{E}\Bigl(\Phi^{(1)}\circ\Phi^{(2)}(x^{(2)}, t^{(2)}), s\frac{\eta_1}{\eta_2}+\alpha_j^l\Bigr)\big(\frac{\eta_1}{\eta_2}\big)^3(1-s)^2 \dif s,
	\eeq
	Finally,  by scaling $\eta_1\rightarrow 2^{-j}\eta_1$, the corresponding map for the amplitude becomes
 \begin{align*}
a_2^{j,l}(x^{(2)},t^{(2)}, \eta)=&a(\Phi^{(1)}\circ\Phi^{(2)}(x^{(2)}, t^{(2)}),2^{-j}\eta_1+\alpha_j^l \eta_2, \eta_2)\chi_{j,l}(\alpha_j^l\eta_2+2^{-j}\eta_1,\eta_2),\\
a_2^{j,l,l'}(x^{(2)},t^{(2)}, \eta)=&a(\Phi^{(1)}\circ\Phi^{(2)}(x^{(2)}, t^{(2)}),2^{-j}\eta_1+\alpha_j^l \eta_2, \eta_2)\chi_{j,l'}(\alpha_j^l\eta_2+2^{-j}\eta_1,\eta_2).
\end{align*}
whose support in $\eta-$variable are  contained respectively in
\begin{equation*}\left\{\begin{aligned}
&D_1=\big\{\eta: |\eta_1/\eta_2|\leqslant \varepsilon_0/4, |\eta_2-1|\leqslant \varepsilon_0\big\},\\
&D_2=\big\{\eta': |\eta_1'/\eta_2'-\alpha_j^{l.l'}|\leqslant \varepsilon_0/4, |\eta_2'-1|\leqslant \varepsilon_0\big\},
\end{aligned}\right.\end{equation*}
where
	$
	\alpha_{j}^{l,l'}=2^j(\eta^{j,l'}_1/\eta^{j,l'}_2-\eta^{j,l}_1/\eta^{j,l}_2)
	$
	and it is easy to find  $|\alpha^{l,l'}_{j}|\geqslant \varepsilon_0$ and  $|\alpha^{l,l'}_{j}|$ does not depend $j,l,l'$. Thus to prove \eqref{aamkdsc}, it suffices to show
	\begin{align}
			&\|\mathcal{T}_{\lambda}^{j,l}f\mathcal{T}_{\lambda }^{j,l,l'}f\|_{L^{5/3}(\R^3)}\\
	\lesssim& (\lambda 2^{-2j})^{1/10}
	\Bigl\|\Bigl(\sum_{\nu:\theta_\nu\subset \theta_{j,l}}|\mathcal{T}_{\lambda}^{\nu}f|^2\Bigr)^{1/2}\Bigr\|_{L^{10/3}(\R^3)}
	\Bigl\|\Bigl(\sum_{\nu':\theta_{\nu'}\subset \theta_{j,l'}}|\mathcal{T}_{\lambda}^{\nu'}f|^2\Bigr)^{1/2}\Bigr\|_{L^{10/3}(\R^3)},
	\label{eq:201}
	\end{align}
	where
	\begin{align*}
	&\mathcal{T}_\lambda^{j,l}f=\sum_{\nu:\theta_{\nu}\subset\theta_{j,l}}\mathcal{T}^\nu_\lambda f,\;\mathcal{T}_\lambda^{j,l,l'}f=\sum_{\nu:\theta_{\nu'}\subset\theta_{j,l'}}\mathcal{T}^{\nu'}_\lambda f,\\
	&\mathcal{T}_{\lambda}^{j,l}f(x^{(2)},t^{(2)})=\int
	e^{i\lambda\phi^{(3)}(x^{(2)},t^{(2)},\eta)}
	a_2^{j,l}(x^{(2)},t^{(2)},\eta)f(\eta)\,\dif \eta,\\
	&\mathcal{T}_{\lambda}^{j,l,l'}f(x^{(2)},t^{(2)})=\int e^{i\lambda \phi^{(3)}(x^{(2)},t^{(2)},\eta)}a_2^{j,l,l'}(x^{(2)},t^{(2)},\eta)f(\eta)\dif \eta,\\
	&\mathcal{T}_{\lambda}^{\nu}f(x^{(2)},t^{(2)})=\int
	e^{i\lambda\phi^{(3)}(x^{(2)},t^{(2)},\eta)}
	a_2^{\nu}(x^{(2)},t^{(2)},\eta)f(\eta)\,\dif \eta,\\
	&a_2^{\nu}(x^{(2)},t^{(2)}, \eta)=a(\Phi^{(1)}\circ\Phi^{(2)}(x^{(2)}, t^{(2)}),2^{-j}\eta_1+\alpha_j^l \eta_2, \eta_2)\chi_{\nu}(\alpha_j^l\eta_2+2^{-j}\eta_1,\eta_2),
	\end{align*}
	and the  phase function reads
	\beq
	\phi^{(3)}(x^{(2)},t^{(2)}, \eta)= 2^{-j}x_1^{(2)}\eta_1+ x_2^{(2)}\eta_2+2^{-2j}t^{(2)}
 \eta_1^2(2\eta_2)^{-1}+ \eta_2\mathcal{E}_{2} \big(x^{(2)},t^{(2)},2^{-j}\eta_1/\eta_2\big).
	\eeq
	
	Let $\{R_{\mu}\}_{\mu}$ be a collection of rectangles of sidelength $\sim 2^{-j}\times 2^{-2j}$ which forms a  partition of
a ball $ B(0, \varepsilon_0)\subset \R^2$.  For each $\mu$, we use $x_{\mu}$ to denote the center of $R_{\mu}$ and  $R_{\mu}^{\varepsilon_0}$ to denote the region $R_{\mu}\times (-\varepsilon_0,\varepsilon_0)$.  For $B(0,\varepsilon_0)\subset \mathbb R^3$, we have
	\beq
	B_{\varepsilon_0}:=B(0,\varepsilon_0)\subset \bigcup_{\mu} R_{\mu}^{\varepsilon_0}.
	\eeq
		Therefore it suffices to show for each  $R_\mu^{\varepsilon_0}$
	\begin{align}
	\nonumber
	&\|\mathcal{T}_{\lambda }^{j,l}f\mathcal{T}_{\lambda }^{j,l,l'}f\|_{L^{5/3}(R_{\mu}^{\varepsilon_0})}\\
	\lesssim &(\lambda 2^{-2j})^{\frac{1}{10}}
	\Bigl\|\Big(\sum_{\nu:\theta_{\nu}\subset \theta_{j,l}}|\mathcal{T}_{\lambda}^{\nu}f|^2\Big)^{\frac12}\Bigr\|_{L^{\frac{10}{3}}(w_{R_{\mu}^{\varepsilon_0}})}\Bigl\|\Big(\sum_{\nu':\theta_{\nu}\subset \theta_{j,l'}}|\mathcal{T}_{\lambda}^{\nu'}f|^2\Big)^{\frac12}\Bigr\|_{L^{\frac{10}{3}}(w_{R_{\mu}^{\varepsilon_0}})},
	\label{eq:66}
	\end{align}
	where the implicit constant is uniform with respect to $\mu$.
Note that
	\beq
	\sum_{\mu} w_{R_{\mu}^{\varepsilon_0}}=w_{B_{\varepsilon_0}},
	\eeq
	then \eqref{eq:201} will follow by squaring both sides of \eqref{eq:66} and Cauchy-Schwarz inequality after summing over all $\mu$'s,
	
	By changing   variables $ x^{(2)}\rightarrow x_\mu+(2^{-j} \tilde{x}_1, 2^{-2j}\tilde{x}_2), t\rightarrow \tilde t$, it suffices to show
	
\begin{align}\nonumber
		& \|\tilde {\mathcal{T}}_{\lambda 2^{-2j}}^{j,l}f \tilde {\mathcal{T}}_{\lambda 2^{-2j}}^{j,l,l'}f\|_{L^{\frac53}(B_{\varepsilon_0})}\\
\lesssim&  (\lambda 2^{-2j})^{\frac{1}{10}}\Big\|(\sum_{\nu:\theta_\nu\subset \theta_{j,l}}|\tilde{\mathcal{T}}_{\lambda 2^{-2j}}^{\nu}f|^2)^{\frac12}\Big\|_{L^{\frac{10}3}(w_{B_{\varepsilon_0}})}\Big\|(\sum_{\nu':\theta_{\nu'}\subset\theta_{j,l'}}|\tilde{\mathcal{T}}_{\lambda 2^{-2j}}^{\nu'}f|^2)^{\frac12}\Big\|_{L^{\frac{10}3}(w_{B_{\varepsilon_0}})},\label{eq:61}
		\end{align}
	 where the phase function in $\tilde {\mathcal{T}}_{\lambda  2^{-2j}}^{j,l}f $
	and $\tilde {\mathcal{T}}_{\lambda  2^{-2j}}^{j,l,l'}f $ reads
	\beq\label{eq:20}
	\tilde{\phi}^{j}(\tilde x, \tilde t,\eta)=\tilde x\eta+\frac{1}{2}\tilde t \eta_1^2/\eta_2 +2^{2j}\eta_2\mathcal{E}_2(x_{\mu}+(2^{-j}\tilde{x}_1,2^{-2j} \tilde{x}_2),t, 2^{-j}\eta_1/\eta_2),
	\eeq
	along with the amplitudes
	\begin{align*}
	&\tilde{a}^{j,l}(\tilde x, \tilde t, \eta )=a_2^{j,l}(x_\mu+(2^{-j} \tilde{x}_1, 2^{-2j}\tilde{x}_2), \tilde t,\eta),\\
	&\tilde{a}^{j,l,l'}(\tilde x, \tilde t, \eta )=a_2^{j,l,l'}(x_\mu+(2^{-j} \tilde{x}_1, 2^{-2j}\tilde{x}_2), \tilde t,\eta).
	\end{align*}
		The constant appearing in \eqref{eq:11} is independent of $j$  which will be clarified later and the finite many of $A_\beta$'s  is uniformly bounded since the error term $\mathcal{E}_{2}$ converges to
	\beq
	\tilde{\phi}_\infty= x\eta+\frac{t}{2}\frac{\eta_1^2}{\eta_2}
	\eeq
	in the sense that
	\beq
	\|\partial_{z}^\beta \tilde{\phi}^j -\partial_z^\beta\phi_\infty\|_{L^\infty}\rightarrow 0, \;  \text{as }\; j\rightarrow \infty,|\beta|\leq N,
	\eeq
	which can be seen from formula \eqref{eq:171}. Using  \eqref{eq:w5d}  with $\lambda$ replaced by $\lambda 2^{-2j}$ we will obtain \eqref{eq:61}.
\end{proof}

\section{Use of the bilinear oscillatory integral estimate}\label{sect4}
In this section, we use the bilinear oscillatory integral estimates to prove \eqref{eq:w5d} and this will complete the proof of \eqref{eq:333'}
due to Proposition \ref{Pro3}.  Let us  define the rescaled function $\phi^\lambda$ and $a_\lambda$ as follows
\beq
\phi^{\lambda}(z,\eta):=\lambda\phi\Bigl(\frac{z}{\lambda},\eta\Bigr),\; \;\;\; a_{\lambda}(z,\eta):=a\Bigl(\frac{z}{\lambda},\eta\Bigr).
\eeq
Correspondingly, we define $\mathscr{T}_{\lambda}$ as
\begin{equation}
\label{mmkmmm}
(\mathscr{T}_\lambda f)(z):=(T_\lambda f)(z/\lambda )=\int e^{i\phi^\lambda(x,t,\eta)}a_\lambda(x,t,\eta) f(\eta)\,\dif\eta.
\end{equation}
We make angular decomposition as before to write $\mathscr{T}_\lambda f$ as
\beq
\mathscr{T}_\lambda f(z)=\sum_{\nu}\mathscr{T}_\lambda^\nu f(z),\;\mathscr{T}_\lambda^\nu f(z)=\int e^{i\phi^\lambda(x,t,\eta)}a_\lambda^\nu(x,t,\eta)f(\eta)\dif\eta,
\eeq
where
the associated amplitude  $a_\lambda^\nu$ for $\mathscr{T}_\lambda^\nu$ is to be $a_\lambda^\nu(z,\eta)=a_\lambda(z,\eta) \chi_\nu(\eta)$.
We rewrite  \eqref{eq:w5d} in the scaled version as follows
\begin{proposition}\label{pro1}
	Let $(\Omega,\Omega')$  satisfy the same angular separation condition in the sense of \eqref{eq:181}. Up to a negligible term, then we have
	\begin{align}\nonumber
	&	\Bigl\|\sum_{\nu \in \Omega}\mathscr{T}_{\lambda}^{\nu}f\sum_{\nu' \in \Omega'}\mathscr{T}_{\lambda}^{\nu'}f\Bigr\|_{L^{\frac53}(\R^{2+1})}\\
	\lessapprox& \lambda^{\frac1{10}}
	\Bigl\|\Big(\sum_{\nu\in \Omega}|\mathscr{T}_{\lambda}^{\nu}f|^2\Big)^{\frac12}\Bigr\|_{L^{\frac{10}3} (\R^{2+1})}
	\Bigl\|\Big(\sum_{\nu' \in \Omega'}|\mathscr{T}_{\lambda}^{\nu'}f|^2\Big)^{\frac12}\Bigr\|_{L^{\frac{10}3} (\R^{2+1})}
	\label{eq:21}\end{align}
\end{proposition}

\begin{definition}[locally constant property]
	For $n\geqslant 1$, given a function $F: \R^n \rightarrow [0,\infty)$, we say $F$ satisfies the
	\emph{ locally constant  property} at scale of $\rho$ if $F(x)\approx F(y)$ whenever $|x-y|\leqslant C_0 \rho$. Here the implicit constant in ``\,$\approx$" could depend on the structure constant $C_0$.
\end{definition}
Let $E$ be  the extension operator given by
\beq
Ef(x,t):=\int_{\R^2} e^{i(x\eta+t h(\eta))} a_2(\eta) f(\eta) \dif \eta,
\eeq
where  $h(\eta)$ is a smooth function away from the origin and homogeneous of degree $1$ with
\begin{equation}
{\rm rank}\;\partial^2_{\eta\eta}{h}=1,\quad \text{for all } \;\; \eta\in {\rm supp}\;a_2.
\end{equation}

For $r\geqslant 1$,  if $f$ is supported in a  $r^{-1}$ neighborhood of $\eta_0 \in {\rm supp} \;a_2$, then ${\rm supp} \;\widehat{Ef}$ is contained  in a ball of radius $r^{-1}$, by uncertainty principle, one may  view $|Ef|$ essentially as a constant at the scale of $r$. However, this is not the case for the oscillatory  operator $\mathscr{T}_{\lambda}$, since  $\widehat{\mathscr{T}_\lambda f}$ is not necessarily compactly supported. One may nevertheless, up to  phase rotation and a negligible term, recover the locally constant property.
\begin{lemma}[\cite{GHI}]\label{le3}
	Let $\mathscr{T}_{\lambda}$ be given by \eqref{mmkmmm}. There exists a smooth, rapidly decreasing function $\varrho:\R^3\to[0,\infty)$ with the following property: ${\rm supp}\; \hat{\varrho} \subset B(0,1)$ such that  if $\e>0$ and $1\leqslant  r \leqslant\lambda^{1-\e}$,
	$f$ is supported in a $r^{-1}$-cube centered at $\bar{\eta}$, then
	\beq\label{eq:47}
	e^{- i \phi^\lambda(z,\bar{\eta})} \mathscr{T}_{\lambda}f(z)=\Bigl(\bigl[e^{- i \phi^{\lambda}(\cdot,\bar{\eta})}
	\mathscr{T}_{\lambda}f(\cdot)\bigr]\ast \varrho_r\Bigr)(z)+{\rm RapDec}(\lambda)\|f\|_{L^{10/3}(\R^2)}
	\eeq
	holds for all $z\in\R^3$, where $\varrho_r(z)=r^{-3}\varrho(z/r)$.
\end{lemma}
\begin{remark}
	We further choose $\varrho$  to satisfy the locally constant property at the scale of {\bf 1}.  Consequently, one may view  $\varrho_r$
	as a constant at scale of $r$.
\end{remark}

To prove Proposition \ref{pro1},  we need further decompose the support of
$\eta\to a_\lambda^\nu(z,\eta)$ in the radial direction to obtain equally-spaced pieces so that we may exploit the locally constant property.
Let $\rho\in C_{c}^{\infty}(\R^2) $   satisfy
\beq\label{eq:ooo}
\sum_{j\in\Z^2}\rho(\eta-j)\equiv 1,\quad \eta\in\R^2.
\eeq

Let $\e>0$ be small and
$\mathbf{Q}=\{Q_k\}_k$ be a mesh of cubes of sidelenghth $ \lambda^{1/2-\e}$,
which are centered at
lattices belong to  $\lambda^{1/2-\varepsilon}\mathbb{Z}^{2+1}$
with sides parallel to the axis and form a tiling of ${\rm supp}_z a(\cdot, \eta)$.
For each $Q_k \in\mathbf{Q}$, let $z_k$ be
the center of $Q_k$
and set
\begin{align*}
\mathscr{T}_{\lambda,\,k}^{\nu,\,j}f(z)
=&\int e^{i\phi^\lambda (z,\eta)}a_{\lambda,k}^{\nu,j}(z,\eta) f(\eta)\,\dif\eta,
\\
a_{\lambda,k}^{\nu,j}(z,\eta)
=&a_{\lambda}^{\nu}(z,\eta)\rho\bigl(\lambda^{1/2-\varepsilon/2}\partial_{x}\phi^\lambda(z_k,\eta)-j\bigr).
\end{align*}
The support of $\eta\to a_{\lambda,k}^{\nu,j}(z,\cdot)$
is  contained in a cube of sidelength $\approx\lambda^{-1/2+\varepsilon/2}$ which is denoted by $\mathcal{D}^{\nu,j}_k$.
Ultimately, we may write
\beq \label{eq:62}
\mathscr{T}_{\lambda}f(z)\Big|_{z\in Q_k}=\sum_{\nu ,j} \mathscr{T}_{\lambda,k}^{\nu,j}f(z),
\quad
\mathscr{T}_{\lambda}^\nu f(z)\Big|_{z\in Q_k}=\sum_{j} \mathscr{T}_{\lambda,k}^{\nu,j}f(z) .
\eeq
Let $\eta^{\nu,j}_k$ be the center of $\mathcal{D}^{\nu,j}_{k}$ and  fix  $R=\lambda^{1/2-\varepsilon/2}$  in what follows.
The key ingredient  in the proof of Proposition \ref{pro1} is the following discrete version of  bilinear estimate.
\begin{proposition}\label{pro2}
	Let $Q_k \in\mathbf{Q}$ be defined as above and $(\nu,\nu')\in\Omega\times\Omega'$
	satisfy the angular
	separation condition \eqref{eq:181}.
	Then, we have
	\begin{equation}
	\label{eq:27}
	\Bigl\|\sum_{\nu,j}e^{i\phi^\lambda(z,\eta_k^{\nu,j})}c^{\nu,j}\sum_{\nu',j'}
	e^{i\phi^\lambda(z,\eta_k^{\nu',j'})}c^{\nu',j'}\Bigr\|_{L^{5/3}(Q_k)}
	\lessapprox \lambda\; \Bigl(\sum_{\nu,j}|c^{\nu,j}|^2\Bigr)^{\frac{1}{2}}
	\Bigl(\sum_{\nu',j'}|c^{\nu',j'}|^2\Bigr)^{\frac{1}{2}}
	\end{equation}
	
\end{proposition}
\begin{proof}
	
	Without loss of generality, we may assume $z_k=\mathbf{0}$  and normalise the phase function by setting
	$
	\psi^{\lambda}(z,\eta)=\phi^\lambda(z,\eta)-\phi^{\lambda}(\mathbf{0},\eta).
	$
	Let
	\beq
	b_k^{\nu,j}(z,\eta)
	=\Bigl(\int_{\mathcal{D}_k^{\nu,j}}
	e^{i[\psi^{\lambda}(z,\eta)-\psi^{\lambda}(z,\eta^{\nu,j}_k)]}\,
	\dif\eta \Bigr)^{-1}\times\chi_{\mathcal{D}_k^{\nu,j}}(\eta),
	\eeq
	where $\chi_{\mathcal{D}_k^{\nu,j}}$ denotes the characteristic function of $\mathcal{D}_k^{\nu,j}$.
	It is easy to see for $\lambda$ sufficiently large\footnote{See the proof of Lemma 4.6 in \cite{MSS-jams} for a similar fact.}
	\begin{gather}
	|b_k^{\nu,j}(z,\eta)|\leqslant R^2,\quad
	|\partial^{\alpha}_z b_k^{\nu,j}(z,\eta)|\leqslant   R^{2-|\alpha|},\quad \forall z\in Q_k \label{eq:33}.
	\end{gather}
	We may evaluate the left side of
	\eqref{eq:27} by
	\begin{equation*}
	\label{eq:67}
	\Big\|\biggl(\int  e^{i\psi^{\lambda}(z,\eta)}\sum_{\nu,j}
	b_k^{\nu,j}(z,\eta)\,c^{\nu,j}\,\dif\eta\biggr)
	\biggl(\int  e^{i\psi^{\lambda}(z,\eta')}\sum_{\nu',j'}
	b_k^{\nu',j'}(z,\eta')\,c^{\nu',j'}\,\dif\eta'\biggr)\Big\|_{L^{\frac{5}{3}}(Q_k)}.
	\end{equation*}	
	Let
	\beq
	B_{k}(z,\eta,\eta')
	=\sum_{\nu,j}\sum_{\nu',j'}c^{\nu,j}\,b_k^{\nu,j}(z,\eta)\;c^{\nu',j'}\,b_k^{\nu',j'}(z,\eta').
	\eeq
	By the fundamental theorem of calculus,
	we may write
	\begin{align}
	\label{eq:29}
	B_{k}(z,\eta,\eta')
	=&B_{k}(\mathbf{0},\eta,\eta')+\int_0^{x_1}\frac{\partial B_k}{\partial u_1}((u_1,0,0),\eta,\eta')\,\dif u_1 \\
	&+\cdots+\int_{0}^{x_1}\int_0^{x_2} \int_{0}^{t}\frac{\partial^{3} B_{k}}{\partial u_1\partial u_2\partial u_{3}}(u,\eta,\eta')\,\dif u.
	\end{align}
	We take  $B_{k}(\mathbf{0},\eta,\eta')$ as an example for the other terms can be handled in a similar way.
	
	The explicit formula of $B_{k}(\mathbf{0},\eta,\eta')$ reads
	\beq
	B_{k}(\mathbf{0},\eta,\eta')=\sum_{\nu,j}c^{\nu,j}b_k^{\nu,j}(\mathbf{0},\eta)\sum_{\nu',j'}c^{\nu',j'}b_k^{\nu',j'}(\mathbf{0},\eta'),
	\eeq
	and we are led to estimating
	\begin{equation}
	\label{kkkkk}
	\Biggl\|\biggl(\int  e^{i \psi^{\lambda}(z,\eta)}
	\sum_{\nu,j}c^{\nu,j}\,b_k^{\nu,j}(\mathbf{0},\eta) \,\dif\eta \biggr)\biggl(\int e^{i \psi^{\lambda}(z,\eta')}\sum_{\nu',j'}c^{\nu',j'}\,
	b_k^{\nu',j'}(\mathbf{0},\eta')\,\dif\eta'\biggr)\Biggr\|_{L^{\frac{5}{3}}}.
	\end{equation}
	According to Theorem \ref{theo2}, \eqref{kkkkk} would be dominated  as
	\begin{align*}
	\Bigl\|\sum_{\nu,j}c^{\nu,j}b_k^{\nu,j}(\mathbf{0},\cdot)\Bigr\|_{L^2}
	\Bigl\|\sum_{\nu',j'}c^{\nu',j'}b_k^{\nu',j'}(\mathbf{0},\cdot) \Bigr\|_{L^2}
	\lessapprox C \lambda \; \Bigl(\sum_{\nu,j}|c^{\nu,j}|^2\Bigr)^{\frac12}\Bigl(\sum_{\nu',j'}|c^{\nu',j'}|^2\Bigr)^{\frac12},
	\end{align*}
	provided the phase functions $\psi^\lambda(z,\eta), \psi^\lambda(z,\eta')$ fulfill the  condition  \eqref{eq:11}.
	
	To see this is the case, and hence complete the proof of\eqref{eq:27},
	we resort to the angular separation condition \eqref{eq:181}.
	In fact,
	after changing variables $z\to \lambda z$,
	it suffices to show that if $\eta,\eta'$ satisfy  \eqref{eq:181}, then \eqref{eq:11} holds for $\psi(z,\eta)$ and $\psi(z,\eta')$  where
	\beq
	\psi(x,t,\eta)=\phi(x,t,\eta)-\phi(0,\eta),
	\eeq
	with $\phi(x,t,\eta)$ of the form \eqref{eq:18}, \emph{i.e.}
	\beq
	\phi(x,t,\eta)=\langle x,\eta\rangle+\frac{t}{2}(\eta_1^2/\eta_2)+\eta_2\mathcal{E}(x,t, \eta_1/\eta_2),
	\eeq
	where the error term $\mathcal{E}(z,s)$ satisfies
	\beq
	\mathcal{E}(z,s)=O(|(x,t)|^2s^2+(|x|+|t|)|s|^3).
	\eeq
	Furthermore,
	we may neglect $\phi(0,\eta)$
	since it is independent of $(x,t)$.
	
	To guarantee \eqref{eq:11}, we need to
	remove the dependence on $\varepsilon_0$ of  the scale of angular separation through  an additional
	angular transformation with respect to the $\eta-$variable. Without loss of generality, we may assume ${\rm supp}_\eta \;a(z,.)$  contains ${\bf e_2}$, since  the general case follows by repeating the argument  in the proof of Proposition \ref{Pro3}.
	
	By changing of variable $\eta_1\rightarrow \varepsilon_0 \eta_1$, we are reduced to the following property for $\eta, \eta'$,
	which are contained in
	\begin{align*}
	\mathscr{D}_1&=\Big\{\eta=(\eta_1,\eta_2)\in\R^2: \Big|\frac{\eta_1}{\eta_2}\Big|\leqslant \frac1{10}, |\eta_2-1|\leqslant \varepsilon_0\Big\},\\
	\mathscr{D}_2&=\Big\{\eta'=(\eta_1',\eta_2')\in\R^2: \Big|\frac{\eta_1'}{\eta_2'}-\frac12\Big|\leqslant \frac1{10}, |\eta_2'-1|\leqslant \varepsilon_0\Big\},
	\end{align*}
	respectively.
	Within the above setting, the phase function now becomes
	\beq
	\label{kjkjkjkj}
	\phi(x,t,\eta)=\varepsilon_0 x_1\eta_1+x_2\eta_2+
	\frac{t}{2}\varepsilon_0^2 \eta_1^2/\eta_2+\eta_2\mathcal{E}(x,t,\varepsilon_0\eta_1/\eta_2).
	\eeq
	
	Let $\{R_{\mu}\}_{\mu}$ be a collection of pairwise disjoint rectangles of sidelength $\varepsilon_0\times \varepsilon_0^2$, covering $B(0,\varepsilon_0)\subset \R^2$.
	For each $\mu$, we use $x_{\mu}$ to denote the center of $R_{\mu}$.
	For simplicity, we use $R_{\mu}^{\varepsilon_0}$ to denote the rectangle $R_{\mu}\times (-\varepsilon_0,\varepsilon_0)$.
	
	Changing  variables $x\rightarrow x_\mu+(\varepsilon_0x_1, \varepsilon_0^2x_2)$, replacing $\lambda$ with $\varepsilon_0^2\lambda$, and after neglecting harmless terms,
	we have \eqref{kjkjkjkj} of the following form in new variables
	\beq\label{eq:202}
	\phi(x,t,\eta)=x\cdot\eta+\frac{t}{2}\eta_1^2/\eta_2 +\varepsilon_0^{-2}\eta_2\mathcal{E}(x_{\mu}+(\varepsilon_0x_1,\varepsilon_0^2x_2),t, \varepsilon_0\eta_1/\eta_2).
	\eeq
	Taking $(x_\mu+(\varepsilon_0 x_1,\varepsilon_0^2x_2),t) \in B(0,2\varepsilon_0)\subset\R^3$ into account,
	by direct calculation for $\varepsilon_0$ small enough, we have
	\begin{align*}\left\{\begin{aligned}
	&\nabla_{x} \phi(x,t,\eta)=\eta+O(\varepsilon_0|\eta|^2)+O(\varepsilon_0|\eta|^3),\\
	&\partial_{x,\eta}^2 \phi(x,t,\eta)={\rm Id}+O(\varepsilon_0|\eta|),\\
	&\partial_{t} \phi(x,t,\eta)=\frac{1}{2}\eta_1^2/\eta_2+O(\varepsilon_0|\eta|^2)+O(\varepsilon_0|\eta|^3).
\end{aligned}\right.
		\end{align*}
		From \eqref{eq:10}, we have
	\[
	\nabla_\eta\partial_{t}\phi(x,t,\eta)=(\nabla_\eta q)(x,t,\partial_{x}\phi(x,t,\eta))\;\partial_{x\eta}^2\phi(x,t,\eta),
	\]
	and hence
	\[
	(\nabla_\eta q)(x,t,\partial_{x}\phi(x,t,\eta))=\left(\eta_1/\eta_2,-\eta_1^2/(2\eta_2^2)\right)+O(\varepsilon_0|\eta|).
	\]
	Therefore, by choosing $\varepsilon_0$ sufficiently small, we have for
	$\eta \in \mathscr{D}_1, \eta' \in \mathscr{D}_2$
	\beq
	\text{Left hand side of} \, \,\eqref{eq:11}\approx |\eta_1/\eta_2-\eta_1'/\eta_2'|^2 +O(\varepsilon_0)\approx 1.
	\eeq
	This verifies \eqref{eq:11}.
\end{proof}

Now we turn to the  proof of Proposition \ref{pro1}.  For given $R=\lambda^{1/2-\varepsilon/2}$, $\varrho_R$ has the locally constant property at scale of $R$.
For $z\in Q_k$, we
denote by
\[
\mathscr{H}^{\nu,j}_{\lambda,k} f(z)=
e^{-i\phi^\lambda(z,\eta^{\nu,j}_k)}\mathscr{T}_{\lambda,k}^{\nu,j}f (z).\]
Based on Lemma \ref{le3} and the compactness of ${\rm supp}_\eta\, a_{\lambda,k}^{\nu,j}(z,\cdot)$ in \eqref{eq:62},
we start with proving the following estiamate
\begin{align}
\label{eq:38} &\sum_{Q_k\in\mathbf{Q}}\Big\|\sum_{\nu,j}
e^{i\phi^{\lambda}(\cdot,\eta^{\nu,j}_k)}
(\mathscr{H}^{\nu,j}_{\lambda,k} f)\ast \varrho_{R}
\sum_{\nu',j'}e^{i\phi^{\lambda}(\cdot,\eta^{\nu',j'}_k)}(\mathscr{H}^{\nu',j'}_{\lambda,k} f)\ast \varrho_{R}\Big\|_{L^{\frac53}(Q_k)}^{\frac53}\\
\lessapprox & \lambda^{\frac{1}{10}}\Big(\sum_{Q_k\in\mathbf{Q}} \Bigl\|(\sum_{\nu,j}|\mathscr{T}_{\lambda,k}^{\nu,j}f|^2)^{\frac{1}{2}}\Bigr\|_{L^{\frac{10}{3}}(w_{Q_k})}^{\frac{10}{3}}\Big)^{\frac{1}{2}}
\Big(\sum_{Q_k\in\mathbf{Q}} \Bigl\|(\sum_{\nu',j'}|\mathscr{T}_{\lambda,k}^{\nu',j'}f|^2)^{\frac{1}{2}}\Bigr\|_{L^{\frac{10}{3}}(w_{Q_k})}^{\frac{10}{3}}\Big)^{\frac{1}{2}}.\nonumber
\end{align}
Indeed, by Minkowski's inequality and
the locally constant property at scale
$R$ enjoyed by
$\varrho_{R}$, we have
\[
\Big\|\iint \varrho_{R}(z-y)
\varrho_{R}(z-y')
\Bigl|\sum_{\nu,\nu',j,j'}
e^{i \phi^\lambda(z,\eta^{\nu,j}_k)}
\mathscr{H}^{\nu,j}_{\lambda,k} f(y)
e^{i\phi^\lambda(z,\eta^{\nu',j'}_k)}\mathscr{H}^{\nu',j'}_{\lambda,k} f(y')\Bigr|\,\dif y\dif y'\Big\|_{L^{\frac53}(Q_k)}
\]
is bounded by
\begin{multline*}
\iint
\Big\| \sum_{\nu,j}
e^{i\phi^\lambda(z,\eta^{\nu,j}_k)}
\mathscr{H}^{\nu,j}_{\lambda,k} f(y)
\sum_{\nu',j'}
e^{i \phi^\lambda(z,\eta^{\nu',j'}_k)}
\mathscr{H}^{\nu',j'}_{\lambda,k} f(y')\Big\|_{L^{\frac{5}{3}}(Q_k)}\\
\times
\varrho_{R}(\bar{z}-y)
\varrho_{R}(\bar{z}-y')\,\dif y\dif y',
\end{multline*}
whenever $\bar{z}\in Q_k$.
We use Proposition \ref{pro2} to obtain the following bound of the above formula,
\begin{equation}
\label{eq:PPP}
\lambda\iint \Bigl(\sum_{\nu,j}|\mathscr{T}_{\lambda,k}^{\nu,j}f(\bar{z}-y)|^2\Bigr)^{\frac12}
\Bigl(\sum_{\nu',j'}|\mathscr{T}_{\lambda,k}^{\nu',j'}f(\bar{z}-y')|^2\Bigr)^{\frac12}
\varrho_{R}(y) \varrho_{R}(y')\,\dif y\dif y'.
\end{equation}
After
averaging over $Q_k$ in $\bar{z}-$variable,
and neglecting the RapDec$(\lambda)$ term as well as  a $\lambda^{O(\e)}$ factor, \eqref{eq:PPP} can be controlled by
\begin{align}
\nonumber\lambda^{\frac1{10}}&\int \Bigl\|\Bigl(\sum_{\nu,j}|\mathscr{T}_{\lambda,k}^{\nu,j}f(\bar{z}-y)|^2\Bigr)^{\frac12}
\Bigr\|_{L^{\frac{10}3}_{\bar{z}}(Q_k)}\varrho_{R}(y)
\dif y\\
&\times
\int \Bigl\|\Bigl(\sum_{\nu',j'}|\mathscr{T}_{\lambda,k}^{\nu',j'}f(\bar{z}-y')|^2\Bigr)^{\frac12}
\Bigr\|_{L^{\frac{10}3}_{\bar{z}}(Q_k)} \varrho_{R}(y')\,\dif y'. \label{mkdcds}
\end{align}
By H\"older's inequality, we have
\[
\int \Bigl\|
\Bigl(\sum_{\nu,j}
|\mathscr{T}_{\lambda,k}^{\nu,j}f(\bar{z}-y)|^2\Bigr)^{\frac{1}{2}}\Bigr\|_{L^{\frac{10}{3}}_{\bar{z}}(Q_k)}
\varrho_{R}(y)
\,\dif y
\lesssim
\Bigl(\int \Bigl(\sum_{\nu,j}|\mathscr{T}_{\lambda,k}^{\nu,j}f(z)|^2\Bigr)^{\frac{5}{3}}w_{Q_k}(z)\, \dif z
\Bigr)^{\frac{3}{10}},
\]
where we have used the following fact,
\[
\int_{\R^3}w_{Q_k}(z+y)\varrho_{R}(y)\,\dif y\lessapprox w_{Q_k}(z).
\]
Summing over $Q_k\in\mathbf{Q}$ and applying Cauchy-Schwarz's inequality, we obtain \eqref{eq:38}.

Assuming that
up to a ${\rm RapDec}(\lambda)-$term,
one may add up the blocks of square functions along radial directions
\beq \label{eq:39}
\Bigl\|\Big(\sum_{\nu,j}
|\mathscr{T}_{\lambda,k}^{\nu,j}f|^2\Big)^{\frac12}
\Bigr\|_{L^{10/3}(w_{Q_k})}
\lessapprox \Bigl\|\Big(\sum_{\nu}|\mathscr{T}_{\lambda}^{\nu}f|^2\Big)^{\frac12}\Bigr\|_{L^{10/3}(w_{Q_k})},
\eeq
we obtain \eqref{eq:21} and this completes the proof of Proposition \ref{pro1}
Thus, it remains to prove \eqref{eq:39} which will be achieved in the next section.

\section{Adding up  blocks  along radial directions}
This section is devoted to showing
\beq
\label{eq:5.1}
\Bigl\|\Big(\sum_{\nu,j}
|\mathscr{T}_{\lambda,k}^{\nu,j}f|^2\Big)^{\frac12}
\Bigr\|_{L^{\frac{10}{3}}(w_{Q_k})}
\lessapprox \Bigl\|\Big(\sum_{\nu}|\mathscr{T}_{\lambda}^{\nu}f|^2\Big)^{\frac12}\bigr\|_{L^{\frac{10}{3}}(w_{Q_k})}+{\rm RapDec}(\lambda)\|f\|_{L^{\frac{10}{3}}(\R^2)}.
\eeq
The main idea is to effectively approximate  $\mathscr{T}_{\lambda}$ by an extension operator $E $  at  suitable small spatial  scale.

Assume $\delta>0$  and $1\leqslant K\leqslant \lambda^{1/2-\delta}$. By taking  Taylor expansion of $\phi^{\lambda}$ around  the point $\bar z$  and changing variables: $\eta \rightarrow \Psi^\lambda(\bar z, \eta):= \Psi(\bar z/\lambda, \eta)$, we have
\beq
\mathscr{T}_{\lambda}f(z)=\int_{\R^2}e^{i(\langle z-\bar z, \partial_{ z}\phi^{\lambda}(\bar z, \Psi^\lambda(\bar z,\eta))\rangle+\varepsilon_{\lambda}^{\bar z}(z-\bar z,\eta))}a_{\lambda, \bar z}(z,\eta) f_{\bar z}(\eta)\dif \eta, \; \text{for}\; |z-\bar z|\leqslant K,
\eeq
where
\begin{align*}
& f_{\bar z}:=e^{i\phi^{\lambda}(\bar z, \Psi^\lambda(\bar z,.))}f\circ \Psi^{\lambda}(\bar z, \cdot), \\
&a_{\lambda,\bar z}(z, \eta)=a_{\lambda }(z, \Psi^{\lambda}(\bar z,\eta))|{\rm det}\,\partial_\eta \Psi^\lambda(\bar z,\eta)|,
\end{align*}
and for $|v|\leqslant K$,
\beq
\varepsilon^{\bar z}_{\lambda}(v,\eta)=\frac{1}{\lambda}\int_0^1 (1-s)\langle (\partial_{zz}^2 \phi)((\bar z+sv)/\lambda, \Psi^{\lambda}(\bar z,\eta))v,v\rangle \dif s.
\eeq
Owing to \eqref{eq:191}, we have for $\lambda\gg 1$
\beq \label{eq:64}
\sup \limits_{(v,\eta)\in B(0,K)\times {\rm supp}_\eta a_{\lambda, \bar z}} |\partial^{\beta}_{\eta}\varepsilon_{\lambda}^{\bar z}(v,\eta)|\leqslant 1, \; \text{for}\; |v|\leqslant K,
\eeq
where  $\beta \in \mathbb{N}^{2}$ and $|\beta|\leqslant N $.

 In view of  \eqref{eq:44},  we obtain
\beq
\langle z, \partial_{ z}\phi^{\lambda}(\bar z, \Psi^\lambda(\bar z,\eta)\rangle =x\eta+t h_{\bar z}(\eta),
\eeq
where $h_{\bar z}(\eta):=(\partial_t \phi^{\lambda})(\bar z, \Psi^\lambda(\bar z, \eta))$.

 Since we assume $a(z,\eta)=a_1(z)a_2(\eta)$, up to the  negligible influence of space variable, heuristically,
we may approximate $\mathscr{T}_{\lambda}$ by extension operators $E_{\bar z}$
\beq\label{eq:45}
E_{\bar z}g(z):=\int_{\R^d}e^{i(x\eta+th_{\bar z}(\eta))}a_{2,\bar z}(\eta) g(\eta)\dif \eta,
\eeq
in a sufficiently small neighborhood of $\bar z$,
where
$$a_{2,\bar z}(\eta)= a_2(\Psi^\lambda (\bar z,\eta))|\,{\rm det}\,\partial_\eta \Psi^\lambda(\bar z,\eta)|.$$
It is obvious that $h_{\bar z}(\eta)$ is homogeneous of degree $1$ and satisfying
\begin{equation}
{\rm rank}\;\partial^2_{\eta\eta}{h}_{\bar z}=1,\quad \text{for all } \;\; \eta\in {\rm supp}\; a_{2,\bar z}.
\end{equation}
Due to the compactness of the support of $a$ and \eqref{eq:50}, we may assume the nonvanishing eigenvalue of $\partial_{\eta\eta}^2h_{\bar z}(\eta)$ is comparable to $1$ independent of $\bar z$.

To show \eqref{eq:5.1}, we shall need the  following two lemmas.
\begin{lemma}
	\label{lem:A}
	Let $z_k$ be the center of $Q_k$
	and  $E_{z_k}^{\nu,j}$ be an extension operator defined by
	\beq
	E_{z_k}^{\nu,j}g(z):=\int_{\R^2}
	e^{i(\langle x,\eta\rangle+th_{z_k}(\eta))}\rho(\lambda^{1/2-\varepsilon/2}\eta-j)a_{2,z_k}^\nu(\eta) g(\eta)\dif \eta
	\eeq
	with $\rho$ satisfying \eqref{eq:ooo} and
	\begin{align*}
	a_{2,z_k}^\nu(\eta)&=a^\nu_{2}(\Psi^{\lambda}(z_k,\eta))|{\rm det}\,\partial_\eta \Psi^\lambda(z_k,\eta)|.
	\end{align*}
	Then we have
	\begin{align}\label{eq:127}
			\Big\|\Big(\sum_{\nu,j}|E_{z_k}^{\nu,j}g|^2\Big)^{\frac12}\Big\|_{L^{\frac{10}{3}}(w_{Q_0})}\lessapprox\Big\|\Big(\sum_{\nu}|E_{z_k}^\nu g|^2\Big)^{\frac12}\Big\|_{L^{\frac{10}{3}}(w_{Q_0})}
+{\rm RapDec(\lambda)}\|g\|_{L^{\frac{10}3}}.
	\end{align}
\end{lemma}

The following lemma shows that when localized in a relatively small region, $\mathscr{T}_\lambda$ is comparable to $E_{\bar{z}} $  in a suitable sense.
A slightly weaker version of this lemma appeared in the work \cite{BelHicSog18P}, which is applicable
to the decoupling norm but is not sufficient to handle square-function estimates.
It is for this reason that we need a pointwise refinement of stability lemma as below.

\begin{lemma} \label{lem:B}
	Let $0<\delta\leqslant 1/2$ and $1\leqslant K\leqslant \lambda^{1/2-\delta}$. Then for any $N$ given by \eqref{eq:191},
\begin{align}	
&\label{eq:128}
	|\mathscr{T}_{\lambda}f(\bar z+v)| \leq  |E_{\bar z}f_{\bar z}(v)|+\Bigl(\frac{3}{\pi}\Bigr)^N \sum_{l\in \mathbb{Z}^2\backslash \{0\}}|l|^{-N} |E_{\bar z}(f_{\bar z}e^{ i  \langle 4\pi l,\cdot\rangle})(v)|\\
	&\label{eq:123aa}
	|E_{\bar z}f_{\bar z}(v)|\leq |\mathscr{T}_{\lambda}f(\bar z+v)|+ \Bigl(\frac{3}{\pi}\Bigr)^N\sum_{l \in \mathbb{Z}^2\backslash \{0\}}|l|^{-N}\big|\mathscr{T}_{\lambda}\big[e^{i\langle 4\pi l,(\partial_x\phi^\lambda)(\bar z,\cdot)\rangle}f\big](\bar z+v)\big|,
	\end{align}
	whenever $|v|\leq K$.
\end{lemma}

We postpone the proof of
Lemma \ref{lem:A} and Lemma \ref{lem:B}
in the next two subsections.
Let us continue the proof of \eqref{eq:5.1}.  Let $z_k$ be the center of $Q_k$, therefore by \eqref{eq:128} we have
	\begin{align}\nonumber
	&\quad \Bigl\|\Big(\sum_{\nu,j}|\mathscr{T}_{\lambda,k}^{\nu,j}f|^2\Big)^{\frac12}\Bigr\|_{L^{\frac{10}3}(w_{Q_k})}\\
	&\leq \Bigl\|\Big(\sum_{\nu,j}|\mathscr{T}_{\lambda,k}^{\nu,j}f|^2\Big)^{\frac12}\chi_{\{|z-z_k|\leq \lambda^{\frac{1}{2}-\frac{\varepsilon}{2}}\}}\Bigr\|_{L^{\frac{10}{3}}(w_{Q_k})}
+\Bigl\|\Big(\sum_{\nu,j}|\mathscr{T}_{\lambda,k}^{\nu,j}f|^2\Big)^{\frac12}\chi_{\{|z-z_k|\geq \lambda^{\frac{1}{2}-\frac{\varepsilon}{2}}\}
	}\Bigr\|_{L^{\frac{10}{3}}(w_{Q_k})}\nonumber\\
	&\lesssim_N \sum_{l\in \mathbb{Z}^2}(1+4\pi|l|)^{-N} \Big\|\Big(\sum_{\nu,j}|E_{z_k}^{\nu,j}(f_{z_k}e^{i\langle 4\pi l,\cdot \rangle})|^2\Big)^{\frac12}\|_{L^{\frac{10}{3}}(w_{Q_0})}+{\rm RapDec(\lambda)}\|f\|_{L^{\frac{10}3}}\label{eq:ppp'}
	\end{align}
	It should be  noted that the cube $Q_0$  appearing in the last inequality of \eqref{eq:128} may be slightly larger than the original one.

	Now using the estimate obtained in \eqref{eq:127}, up to a negligible error term, we have
	\begin{align*}
	&\quad \sum_{l}(1+4\pi|l|)^{-N}\Big\|\Big(\sum_{\nu,j}|E_{z_k}^{\nu,j}(f_{z_k}e^{i\langle 4\pi l,\cdot\rangle})|^2\Big)^{\frac12}\Big\|_{L^{\frac{10}3}(w_{Q_0})}\\
	&\lessapprox   \sum_{l}(1+4\pi|l|)^{-N}\Big\|\Big(\sum_{\nu}|E_{z_k}^\nu(f_{z_k}e^{i\langle 4\pi l,\cdot\rangle})|^2\Big)^{\frac12}\Big\|_{L^{\frac{10}3}(w_{Q_0})}\\
	&\lessapprox   \sum_{l}(1+4\pi|l|)^{-N}\Big\|\Big(\sum_{\nu}|E_{z_k}^{\nu}f_{z_k}|^2\Big)^{\frac12}\Big\|_{L^{\frac{10}3}(w_{Q_{0},l})},
	\end{align*}
where $w_{Q_{0},l}(z)=w_{Q_0}((4\pi l,0)+z)$.

	The last inequality  we have used  the fact the extension operator $E_{\bar z}$ is invariant under translation transformation
	\beq
	E_{\bar z}[e^{i\langle 4\pi l,\cdot\rangle}g](x,t)=E_{\bar z}g(x+4\pi l, t).
	\eeq
	Since
	\beq
	\sum_{l}(1+4\pi|l|)^{-N} w_{Q_0,l}(z)\lesssim w_{Q_0}(z),
	\eeq
 we obtain that
	\beq \label{eq:51}
	\Big\|\Big(\sum_{\nu,j}|\mathscr{T}_{\lambda,k}^{\nu,j}f|^2\Big)^{\frac12}\Big\|_{L^{\frac{10}3}(w_{Q_k})}\lessapprox  \Big\|\Big(\sum_{\nu}|E_{z_k}^\nu f_{z_k}|^2\Big)^{\frac12}\Big\|_{L^{\frac{10}3}(w_{Q_0})}+{\rm RapDec(\lambda)}\|f\|_{\frac{10}3}.
	\eeq
	To finish the proof, it suffices to replace  $E_{z_k}^\nu f_{z_k}$ by its variable coefficient counterpart.
	
	From \eqref{eq:123aa} and Minkowski's inequality  we have
	\begin{align*}
	&\;\;\; \Big\|\Big(\sum_{\nu}|E_{z_k}^\nu f_{z_k}|^2\Big)^{\frac12}\|_{L^{\frac{10}3}(w_{Q_0})}\\
	&\leq \Big\|\Big(\sum_{\nu}|E_{z_k}^\nu f_{z_k}|^2\Big)^{\frac12}\chi_{\bigl\{|z|\leq \lambda^{\frac12-\frac{\varepsilon}{16}}\bigr\}}\Big\|_{L^{\frac{10}{3}}(w_{Q_0})}
+ \Big\|\Big(\sum_{\nu}|E_{z_k}^{\nu} f_{z_k}|^2\Big)^{\frac12}\chi_{ \bigl\{|z|\ge \lambda^{\frac12-\frac{\varepsilon}{16}}\bigr\}} \Big\|_{L^{\frac{10}3}(w_{Q_0})}\\
	&\lesssim_N  \sum_{l \in \mathbb{Z}^2}(1+4\pi|l|)^{-N}\Big\|\Big(\sum_{\nu}|\mathscr{T}_{\lambda}^{\nu}(e^{i\langle 4\pi l, \partial_{x} \phi^\lambda(z_k,\cdot)\rangle}f)(z_k+v)|^2\big)^{\frac12}\Big\|_{L^{\frac{10}3}(w_{Q_0})}+{\rm RapDec(\lambda)}\|f\|_{L^{\frac{10}3}}
	\end{align*}
	Note that  when $l=0$, that is what we desire,  thus it remains to control the error term. By \eqref{eq:128},
	\begin{align*}
	&\quad \sum_{l \in \mathbb{Z}^2\backslash \{0\}}\Bigl(\frac{\pi|l|}{3}\Bigr)^{-N}\Big\|\Big(\sum_{\nu}|\mathscr{T}_{\lambda}^{\nu}(e^{i\langle 4\pi l, \partial_{x} \phi^\lambda(z_k,\cdot)\rangle}f)(z_k+v)|^2\Big)^{\frac12}\Big\|_{L^{\frac{10}3}(w_{Q_0})}\\
	&\leq \sum_{l\in \mathbb{Z}^2 \backslash \{0\}} \sum_{k\in \mathbb{Z}^2\backslash \{0\} } \Bigl(\frac{\pi|l|}{3}\Bigr)^{-N}\Bigl(\frac{\pi|k|}{3}\Bigr)^{-N}\Big\|\Big(\sum_{\nu}|E_{z_k}^\nu f_{z_k}|^2\Big)^{\frac12}\Big\|_{L^{\frac{10}3}(w_{Q_0} ((4\pi (l+k),0)+\cdot))}\\
	&\quad +\sum_{l\in \mathbb{Z}^2 \backslash \{0\}}  \Bigl(\frac{\pi|l|}{3}\Bigr)^{-N}\Big\|\Big(\sum_{\nu}|E_{z_k}^\nu f_{z_k}|^2\Big)^{\frac12}\Big\|_{L^{\frac{10}3}(w_{Q_0} ((4\pi l,0)+\cdot))}\\
	&\leq \frac{1}{2}\Big\|\Big(\sum_{\nu}|E_{z_k}^\nu f_{z_k}|^2\Big)^{\frac12}\Big\|_{L^{\frac{10}3}(w_{Q_0} )},
	\end{align*}
	the last inequality can be ensured by presetting  $N$ sufficiently large.
		Therefore we combine the above estimate together
	\begin{align*}
	&\Big\|\Big(\sum_{\nu}|E_{z_k}^\nu f_{z_k}|^2\Big)^{\frac12}\Big\|_{L^4(w_{Q_0})}\\
 \leq& \Big\|\Big(\sum_{\nu}|\mathscr{T}_{\lambda}^{\nu}f|^2\Big)^{\frac12}\Big\|_{L^4(w_{Q_k})}
 + \frac{1}{2}\Big\|\Big(\sum_{\nu}|E_{z_k}^\nu f_{z_k}|^2\Big)^{\frac12}\Big\|_{L^4(w_{Q_0} )}+{\rm RapDec(\lambda)}\|f\|_{L^{\frac{10}3}}.
	\end{align*}
	The term $\|(\sum_{\nu}|E_{z_k}^\nu f_{z_k}|^2)^{1/2}\|_{L^4(w_{Q_0} )}$ appearing in the right hand side
can be absorbed in the left hand side, then we complete the proof.

\subsection{Proof of Lemma \ref{lem:A}}
We shall need the following two lemmas.
The first one is due to Rubio de Francia \cite{Ru83}, which handles the square function estimate for equally-spaces cubes in frequency space.
\begin{lemma}\label{le8}
	Let $\{O_k\}_{k}$ be a collection of equally spaced cubes, and let $\varphi_k(\xi)=\varphi(\xi-\xi_k)$ be the bump function adapted to $O_k$, where $\xi_k$ denotes the center of $O_k$. Then for any function $f$ we have the pointwise estimate
	\beq
	\Bigl(\sum_{k}|\widehat{\varphi}_k\ast f|^2\Bigr)^{1/2}\leq C(\varphi)({\mathbf M}[|f|^2])^{1/2}
	\eeq
	where ${\mathbf{M} }$ denotes the Hardy-Littlewood maximal function, and $C(\varphi)$ depends only on the dimension  and finitely many of the derivatives of $\varphi$ which is associated with the unit cube.
\end{lemma}
The next lemma about the vector-valued maximal function estimate is due to Fefferman- Stein \cite{FS} (see also \cite{Grafakos14A}).
\begin{lemma}\label{le6}
	Let $1<r,p<\infty$, $\{f_k\}_k$ be a sequence of functions, then
	\beq
	\Bigl\|\Big(\sum_{k} |{\mathbf{M}}f_k|^r\Big)^{1/r}\Bigr\|_{L^p(\R^n)}\leq C_n\;A_{p,r} \Bigl\|\Big(\sum_{k} |f_k|^r\Big)^{1/r}\Bigr\|_{L^p(\R^n)},
	\eeq
	where
	$
	A_{p,r}=\frac{r}{r-1}\Bigl(p+\frac{1}{p-1}\Bigr)
	$.
\end{lemma}
The crux of the problem is that the weight function $w_{Q_0}$ is not an $A_p$ weight. Specifically, the following  estimate for the  maximal operator
\beq
\|\mathbf{ M} f\|_{L^p(w_{Q_0})}\leq C\|f\|_{L^p(w_{Q_0})}, \;1<p<\infty,
\eeq
fails. In order to overcome this failure and preserve some kind of localized property, we need to make a series of localization reduction.

Indeed
\begin{align*}
&\Big\|\Big(\sum_{\nu,j}|E_{z_k}^{\nu,j}g|^2\Big)^{\frac{1}{2}}\Big\|_{L^{\frac{10}3}(w_{Q_0})}\\
\leq& \Big\|\Big(\sum_{\nu,j}|E_{z_k}^{\nu,j}g|^2\Big)^{\frac{1}{2}}\chi_{\bigl\{|x|\leq \lambda^{1/2-\varepsilon/4}\bigr\}}\Big\|_{L^{\frac{10}3}(w_{Q_0})}
+\Big\|\Big(\sum_{\nu,j}|E_{z_k}^{\nu,j}g|^2\Big)^{1/2}\chi_{\bigl\{|x|> \lambda^{1/2-\varepsilon/4}\bigr\}}\Big\|_{L^{10/3}(w_{B_0})}.
\end{align*}
Due to the fast decay of the weight $w_{Q_0}$ away from $|z|\geq \lambda^{1/2-\varepsilon/4}$, it suffices to consider
\beq \label{eq:124}
\Big\|\Big(\sum_{\nu,j}|E_{z_k}^{\nu,j}g|^2\Big)^{\frac12}\chi_{\bigl\{|x|\leq \lambda^{1/2-\varepsilon/4}\bigr\}}\Big\|_{L^{10/3}(w_{Q_0})}.
\eeq
Freeze $t_0$ and note that
\beq \label{eqc}
E_{z_k}^{\nu,j}g(x, t_0)=\int  e^{i\langle x,\eta\rangle }\rho(\lambda^{\frac12-\frac{\varepsilon}2}\eta-j)(E_{z_k}^{\nu} g)^{\wedge}(\eta, t_0) \dif \eta
\eeq
where $E_{z_k}^{\nu} g(x,t_0)$ is defined by
\beq
E_{z_k}^{\nu} g(x,t_0):=\int_{\R^2}e^{i(\langle x,\eta\rangle+t_0h_{z_k}(\eta))}a_{z_k}^\nu(\eta) g(\eta)d\eta.
\eeq
We further decompose
\begin{align}\label{eq:123}
	E_{z_k}^{\nu}g(x,t_0)=\chi_{\bigl\{|x|\leq \lambda^{\frac12-\frac{\varepsilon}{8}}\bigr\}}(x)E_{z_k}^{\nu}g(x,  t_0)+\chi_{\bigl\{|x|> \lambda^{\frac12-\frac{\varepsilon}{8}}\bigr\}}(x)E_{z_k}^{\nu}g(x,t_0).
\end{align}
It remains to estimate
\beq \label{eqd}
\int  e^{i\langle x,\eta\rangle }\rho(\lambda^{\frac12-\frac{\varepsilon}{2}}\eta-j)\Big(\chi_{\bigl\{|x|\leq \lambda^{\frac12-\frac{\varepsilon}{8}}\bigr\}}(\,\cdot\,)E_{z_k}^{\nu}g\Big)^{\wedge}(\eta,t_0) \dif \eta.
\eeq
In fact  for $|x|\leq \lambda^{1/2-\varepsilon/4} $, the contribution of the second term in \eqref{eq:123} to \eqref{eqc} can be controlled by
\beq
\int \frac{1}{\lambda^{1-\varepsilon}}\hat{\rho}\Bigl(\frac{x-y}{\lambda^{1/2-\varepsilon/2}}\Bigr)\chi_{\bigl\{|x|> \lambda^{\frac 12-\frac{\varepsilon}{8}}\bigr\}}(y)E_{z_k}^{\nu}g(y) dy \lesssim {\rm RapDec(\lambda)}\|g\|_{L^{10/3}}.
\eeq
We continue the estimate of  \eqref{eqd}, by Lemma \ref{le8}
\begin{align}
\nonumber
&\Big(\sum_{\nu,j}\Big|\int  e^{i\langle x,\eta\rangle }\rho(\lambda^{\frac12-\frac{\varepsilon}2}\eta-j)\Big(\chi_{\bigl\{|\,\cdot\,|\leq \lambda^{\frac12-\frac{\varepsilon}{8}}\bigr\}}(\,\cdot\,)E_{z_k}^{\nu}g\Big)^{\wedge}(\eta,t_0) \dif \eta\Big|^2\Big)^{\frac12}\\
\leq &C \Big(\sum_{\nu}{\rm  M}\Big[ |\chi_{\bigl\{|\,\cdot\,|\leq \lambda^{\frac12-\frac{\varepsilon}{8}}\bigr\}}(\,\cdot\,)E_{z_k}^{\nu}g(\cdot, t_0) |^2\Big]\Big)^{\frac12}.\label{eq:125}
\end{align}
Now unfreezing $t_0$, after plugging the estimate \eqref{eq:125} into \eqref{eq:124} and using Lemma \ref{le6}, we obtain
\begin{align}
\Big\|\Big(\sum_{\nu,j}|E_{z_k}^{\nu,j}g|^2\Big)^{\frac12}\Big\|_{L^{10/3}(w_{Q_0})}
&\lesssim \Big\| \Big(\sum_{\nu} {\rm M} \Big[|\chi_{\bigl\{|x|\leq \lambda^{\frac12-\frac{\varepsilon}{8}}\bigr\}}(\cdot)E_{z_k}^{\nu}g(\cdot, t_0)|^2\Big]\Big)^{\frac12}\|_{L^{\frac{10}3}(w_{Q_0})}\nonumber\\
&\lessapprox \Big\|\Big(\sum_{\nu}|E_{z_k}^\nu g|^2\Big)^{\frac12}\|_{L^{\frac{10}3}(w_{Q_0})}
\label{eq:126}
\end{align}
\begin{remark}
	In the last inequality, additional $\lambda^\varepsilon$ appears in the course of using Fefferman-Stein's square function estimate.  We actually use the following inequality which can be deduced directly from Lemma \ref{le6}
	$$	\Big\|\sum_m  {\rm M} g_m\Big\|_{L^{10/3} (\R^3)}\lesssim_\varepsilon N^{\varepsilon}\Big\|\sum_m |g_m|\Big\|_{L^{10/3}(\R^3)}
	$$
	where $\#\{m\}=N$ and arbitrary $\varepsilon>0$.
	
	In fact,  by H\"older's inequality and Lemma \ref{le6}, we have
	\begin{align*}
	\Big\|\sum_m  {\rm M} g_m\Big\|_{L^{\frac{10}3} (\R^3)}&\leq \Big\|\Big(\sum_m  |{\rm M} g_m|^r\Big)^{1/r}\Big\|_{L^{\frac{10}3}(\R^3)} (\#\{m\})^{1/r'}\\
	&\lesssim_{r}  \Big\|\Big(\sum_m  |g_m|^r\Big)^{1/r}\Big\|_{L^{\frac{10}3}(\R^3)} \big(\#\{m\}\big)^{\frac1{r'}}\\
	&\lesssim_{r} \Big\|\sum_m |g_m|\Big\|_{L^{\frac{10}3}(\R^3)} \big(\#\{m\}\big)^{\frac1{r'}}
	\end{align*}
	We may choose $r'$ sufficiently large such that $\frac{1}{r'}\leq \varepsilon$.
\end{remark}

\subsection{Proof of Lemma \ref{lem:B}}
%
\begin{proof}
	Noting that  ${\rm supp}_\eta a_{\lambda,\bar{z}}(z,\cdot)\subset B(\mathbf{e}_2,\varepsilon_0)$ where $\varepsilon_0$ can be chosen small if necessary,
	we may replace $f$ with $f \psi$, where $\psi$ is a smooth function that equals to $1$ on $B({\bf e_2}, \frac{1}{100})$ and vanishing outside of $B\bigl({\bf e_2}, \frac{1}{5}\bigr)$ such that
	\beq
	|\partial_{\eta}^\alpha \psi(\eta)|\leqslant 6^N, \; \text{for} \; \alpha \in \mathbb{N}^2, 1\leqslant |\alpha|\leqslant N.
	\eeq
	By performing a Fourier expansion  of $e^{i\varepsilon^{\bar z}_{\lambda}(v,\eta)}\psi(\eta)$ in $\eta$ variable, one may write
	\beq
	e^{i\varepsilon^{\bar z}_{\lambda}(v,\eta)}\psi(\eta)=\sum_{l\in \mathbb{Z}^2}b_l(v)e^{i\langle 4\pi l,\eta\rangle},
	\eeq
	where
	\beq
	b_{l}(v)=2\int_{Q({\bf e}_2, 1/4)}e^{-i\langle 4\pi l,\eta\rangle}e^{i\varepsilon^{\bar z}_\lambda(v,\eta)}\psi(\eta)\dif \eta
	\eeq
	where $Q({\bf e}_2, 1/4)$ denotes the cube centered at ${\bf e}_2$ with sidelength $1/2$.
	
	By \eqref{eq:64}, it is easy to show $|b_0(v)|\leqslant 1$. Integration by parts show that
	\beq
	|b_l(v)|\leqslant  12^N (4\pi|l|)^{-N} \quad \text{whenever}\quad |v| \leqslant  K, l\neq (0,0),
	\eeq
	which leads to \eqref{eq:128}.

	For the reverse  direction,  one may write
	\beq
	E_{\bar z}f_{\bar z}(v)=\int_{\R^2}e^{i\phi^\lambda(\bar z+v,\Psi^{\lambda}(\bar z,\eta))} e^{-i\varepsilon^{\bar z}_\lambda(v,\eta)}a_{2,\bar z}(\eta)f\circ \Psi^\lambda(\bar z,\eta)\dif \eta.
	\eeq
	Performing the Fourier expansion of $e^{-i\varepsilon^{\bar z}_\lambda(v,\eta)} $ in $\eta$ and reversing the change of variables $\eta\rightarrow \Psi^\lambda(\bar z,\eta)$, we have
	\eqref{eq:123aa}.
\end{proof}

\section{Comments on higher dimensional cases} \label{sect5}
The results in Proposition \ref{pro1} can be generalized to higher dimensions with an additional \emph{convexity} assumption on the phase function which becomes superfluous in $2+1$ dimensions since  there  exists  only  one non-vanishing eigenvalue. This assumption essentially makes the separation condition \eqref{eq:11} justifiable.

As the situation described in $\R^{2+1}$,  let $n\geq 3$, $a(z,  \eta)\in C_c^\infty(\R^n \times \R^{n+1})$ with compact support contained in $B(0,\varepsilon_0)\times B({\bf e_{n}},\varepsilon_0)$. Assume $\mathcal{C}({\bf\mathbf {e}_n},\varepsilon_0)
:=B({\bf \mathbf{e}_n},\varepsilon_0)\cap\mathbb{S}^{n-1}$ and
make angular decomposition with respect to the $\eta$-variable by cutting
$\mathcal{C}({\bf \mathbf{e}_n},\varepsilon_0)$ into
$N_\lambda\approx\lambda^{(n-1)/2}$ many caps
$\{\theta_{\nu}: 1 \leqslant \nu \leqslant N_ \lambda\}$,
each cap  $\theta_{\nu}$  extends $ \approx_{\varepsilon_0}\lambda^{-(n-1)/2}$.
We denote by $\kappa_\nu\in \mathbb{S}^{n-1}$  the center of $\theta_{\nu}$.

Let $\{\chi_{\nu}(\eta)\}$ be a family of smooth  cutoff function associated with the decomposition in the angular direction,
each of which is  homogeneous of degree $0$,
such that $\{\chi_\nu\}_{\nu}$ forms a partition  of unity on the unit circle
and then extended homogeneously to $\R^n\setminus 0$ such that
\begin{equation*}\left\{\begin{aligned}
&\sum_{0\leqslant \nu\leqslant N_\lambda} \chi_{\nu}(\eta)\equiv 1,\;\;\forall \eta \in \mathbb{R}^n\setminus 0,\\
&|\partial^{\alpha} \chi_{\nu}(\eta)|\leqslant C_\alpha \lambda^{\frac{|\alpha|}{2}},\;\; \forall \;\alpha \; \text{if}\; |\eta|=1.\end{aligned}\right.\end{equation*}
Define
\begin{equation}\label{add-100}\left\{\begin{aligned}
&T_\lambda f=\int e^{i\lambda \phi(z,\eta)}a(z,\eta) f(\eta)\,\dif\eta=\sum_\nu T_\lambda^\nu f,\\
&T_\lambda^\nu f(z)=\int e^{i\lambda\phi(z,\eta)}
a^{\nu}(z,\eta)f(\eta)\,\dif \eta,
\end{aligned}\right.\end{equation}
where $a^{ \nu}(z,\eta)=\chi_\nu(\eta)a(z,\eta)$.

Let $n\geqslant 3$, by  carrying  over the approach in the proof of Proposition \ref{Pro3}, one may obtain, under the similar condition of that in \ref{Pro3} and the convexity condition
\begin{align}
&\quad\bigl\|\sum_{\nu \in \Omega} T_{\lambda}^{\nu}g\sum_{\nu'\in \Omega'} T_{\lambda}^{\nu'}h\bigr\|_{L^{\frac{n+3}{n+1}}(\R^{n+1})}\nonumber\\
&\lessapprox_{\phi,\varepsilon} \lambda^{\frac{n-1}{2(n+3)}}\Bigl\|\Big(\sum_{\nu\in \Omega}|T_{\lambda}^{\nu}g|^2\Big)^{1/2}\Bigr\|_{L^{\frac{2(n+3)}{n+1}}(\R^{n+1})}\Bigl\|\Big(\sum_{\nu'\in \Omega'}|T_{\lambda}^{\nu' }h|^2\Big)^{\frac12}\Bigr\|_{L^{\frac{2(n+3)}{n+1}} (\R^{n+1})},
\label{eq:w6d}
\end{align}
which implies the square function estimate
\beq
\|T_{\lambda}f\|_{	L^{\frac{2(n+3)}{n+1}}(\R^{n+1})}
\lessapprox \lambda^{\frac{n-1}{4(n+3)}}
\Bigl\|\Bigl(\sum_{\nu}
|T_{\lambda}^{\nu} f|^2\Bigr)^{\frac12}\Bigr\|_{L^\frac{2(n+3)}{n+1}(\R^{n+1})}
\eeq
up to a ${\rm RapDec}(\lambda)$ term.

Unfortunately, we are unable to obtain a better result than interpolation between the sharp $L^{q_n}$ estimates of \cite{BelHicSog18P} with the $L^2$ estimate. One of the reason responsible for this shortage is due to the poor knowledge on $L^p\to L^p$ estimates for the variable coefficient version of Kakeya type maximal function in the light ray directions
\begin{equation}
\label{eq:kky-p}
\|\mathcal{M}_\delta\|_{L^p\to L^p}\leqslant C 	\max\Bigl\{\Bigl(\log\frac{1}{\delta}\Bigr)^{\frac{1}{2}},\,\delta^{-\frac{n-2}{p}}\Bigr\},\;\text{for } 2\leqslant p\leqslant\infty,
\end{equation}
when $n\geqslant 3$.
The $L^p-$Kakeya maximal function estimate for $p>2$  is known also for its profundity and difficulty in the literature, for which we refer to \cite{SoggeBook}.
In conclusion, it seems very difficult to adopt the bilinear method used in this paper to refine the result in \cite{MSS-jams}
for $p\leqslant q_n$ when $n\geqslant 3$.

\subsection*{Acknowledgements} The authors were supported by NSFC Grants 11831004.    We are grateful to Christopher Sogge
 for his invaluable comments and suggestions which helped improve the paper greatly.


\bibliographystyle{amsplain}

\end{document}